\documentclass{amsart}
\usepackage{amssymb}
\usepackage{graphicx}
\usepackage{amscd}
\usepackage{color}

\allowdisplaybreaks
\numberwithin{figure}{section}
\numberwithin{table}{section}
\numberwithin{equation}{section}

\newtheorem{thm}{Theorem}[section]

\newtheorem{lem}[thm]{Lemma}
\newtheorem{cor}[thm]{Corollary}
\newtheorem{remark}{Remark}[section]
\newcommand\C{\mathbb{C}}
\newcommand\N{\mathbb{N}}
\newcommand\R{\mathbb{R}}
\newcommand\Z{\mathbb{Z}}

\newcommand\B{\mathcal B}
\newcommand\Ab{\boldsymbol A}
\newcommand\Bb{\boldsymbol B}
\newcommand\D{\mathcal D}
\newcommand\w{\omega}
\renewcommand\Im{\operatorname{Im}}
\renewcommand\Re{\operatorname{Re}}
\newcommand\err{\operatorname{err}}

\newcommand\Lip{\operatorname{Lip}}

\newcommand{\be}{b}

\newcommand{\irm}{\mathrm{i}}
\newcommand{\interior}{\operatorname{int}}

\begin{document}
 \title[Computation of Hausdorff Dimension]
{A New Approach to Numerical Computation of Hausdorff Dimension of Iterated
  Function Systems: Applications to Complex Continued Fractions}
\author{Richard S. Falk}
\address{Department of Mathematics,
Rutgers University, Piscataway, NJ 08854}
\email{falk@math.rutgers.edu}
\urladdr{http://www.math.rutgers.edu/\char'176falk/}

\author{Roger D. Nussbaum}
\address{Department of Mathematics,
Rutgers University, Piscataway, NJ 08854}
\email{nussbaum@math.rutgers.edu}
\urladdr{http://www.math.rutgers.edu/\char'176nussbaum/}
\thanks{The work of the second author was supported by
NSF grant DMS-1201328.}
\date{August 22, 2017}

\begin{abstract}

  In a previous paper \cite{hdcomp1}, the authors developed a new approach to
  the computation of the Hausdorff dimension of the invariant set of an
  iterated function system or IFS and studied some applications in one
  dimension. The key idea, which has been known in varying degrees of
  generality for many years, is to associate to the IFS a parametrized family
  of positive, linear, Perron-Frobenius operators $L_s$. In our context, $L_s$
  is studied in a space of $C^m$ functions and is not compact. Nevertheless,
  it has a strictly positive $C^m$ eigenfunction $v_s$ with positive
  eigenvalue $\lambda_s$ equal to the spectral radius of $L_s$. Under
  appropriate assumptions on the IFS, the Hausdorff dimension of the invariant
  set of the IFS is the value $s=s_*$ for which $\lambda_s =1$.  To compute
  the Hausdorff dimension of an invariant set for an IFS associated to complex
  continued fractions, (which may arise from an infinite iterated function
  system), we approximate the eigenvalue problem by a collocation method using
  continuous piecewise bilinear functions.  Using the theory of positive
  linear operators and explicit a priori bounds on the partial derivatives of
  the strictly positive eigenfunction $v_s$, we are able to give rigorous
  upper and lower bounds for the Hausdorff dimension $s_*$, and these bounds
  converge to $s_*$ as the mesh size approaches zero. We also demonstrate by
  numerical computations that improved estimates can be obtained by the use of
  higher order piecewise tensor product polynomial approximations, although
  the present theory does not guarantee that these are strict upper and lower
  bounds. An important feature of our approach is that it also applies to the
  much more general problem of computing approximations to the spectral radius
  of positive transfer operators, which arise in many other applications.

\end{abstract}

\maketitle

\section{Introduction}
\label{sec:intro}

Our interest in this paper is in describing methods which give rigorous
estimates for the Hausdorff dimension of invariant sets for (possibly
infinite) iterated function systems or IFS's. For simplicity, we do not
consider here the important case of graph directed iterated function systems,
for which a similar approach can be given. Our immediate application is to the
case of invariant sets for IFS's associated to complex continued fractions,
but we expect to show in future work that other interesting examples can also
be treated. In previous work \cite{hdcomp1}, we considered IFS's in one
dimension, and in particular the computation of the Hausdorff dimension of
some Cantor sets arising from continued fraction expansions and also other
examples in which the underlying maps have less regularity.

To describe our present results, let $D \subset \R^n$ be a nonempty compact set,
$\rho$ a metric on $D$ which gives the topology on $D$, and $\theta_{\be}: D
\to D$, $\be \in \B$, a contraction mapping, i.e., a Lipschitz mapping (with
respect to $\rho$) with Lipschitz constant $\Lip(\theta_{\be})$, satisfying
$\Lip(\theta_{\be}):= c_{\be} <1$.  If $\B$ is finite and the above assumption
holds, it is known that there exists a unique, compact, nonempty set $C
\subset D$ such that $C = \cup_{\be \in \B} \theta_{\be}(C)$.  The set $C$ is
called the invariant set for the IFS $\{\theta_{\be} : b \in \B\}$. If
$\B$ is infinite and $\sup \{c_{\be} : \be \in \B\} = c <1$, there is a
naturally defined nonempty invariant set $C \subset D$ such that $C =
\cup_{\be \in \B} \theta_{\be}(C)$, but $C$ need not be compact.  In
\cite{hdcomp1}, the index set $\B$ was finite and could be simply described by
the notation $\theta_j$, $j = 1, \ldots, m$.  In the case of complex continued
fractions, which we consider here, $\be = m + ni$, $m$ belonging to a subset
of $\N$ and $n$ belonging to a subset of $\Z$.

Although we shall eventually specialize, since the method we consider has
applications other than the one we describe in this paper, it is useful, as
was done in \cite{hdcomp1}, to describe initially some function analytic
results in the generality of the previous paragraph. Let $H$ be a bounded,
open, mildly regular (defined in Section~\ref{sec:exist}) subset of $\R^n$ and
let $C^k_{\C}(\bar H)$ denote the complex Banach space of $C^k$ complex-valued
maps, all of whose partial derivatives of order $\nu \le k$ extend
continuously to $\bar H$.  For a given positive integer $N$, assume that
$g_{\be}: \bar H \to (0, \infty)$ are strictly positive $C^N$ functions for
$\be \in \B$ and $\theta_{\be}: \bar H \to \bar H$, $\be \in \B$, are $C^N$
maps and contractions.  For $s >0$ and integers $k$, $0 \le k \le N$, one can
define a bounded linear map $L_{s,k}: C^k(\bar H) \to C^k(\bar H)$ by the
formula
\begin{equation}
\label{intro1.2}
(L_{s,k} f)(x) = \sum_{\be \in \B} [g_{\be}(x)]^s f(\theta_{\be}(x)).
\end{equation}
Note that \eqref{intro1.2} also defines a bounded linear map of $C^k_{\R}(\bar
H)$ to itself, which (abusing notation), we shall also denote by $L_{s,k}$.
Linear maps like $L_{s,k}$ are sometimes called positive transfer operators or
Perron-Frobenius operators and arise in many contexts other than computation
of Hausdorff dimension: see, for example, \cite{Baladi}. If $r(L_{s,k})$
denotes the spectral radius of $L_{s,k}$, then $\lambda_s = r(L_{s,k})$ is
positive and independent of $k$ for $0 \le k \le N$; and $\lambda_s$ is an
algebraically simple eigenvalue of $L_{s,k}$ with a corresponding unique,
normalized strictly positive eigenfunction $v_s \in C^N(\bar H)$.
Furthermore, the map $s \mapsto \lambda_s$ is continuous.  If $\sigma(L_{s,k})
\subset \C$ denotes the spectrum of $L_{s,k}$, $\sigma(L_{s,k})$ depends on
$k$, but for $1 \le k \le N$,
\begin{equation}
\label{intro1.3}
\sup\{|z|: z \in \sigma(L_{s,k})\setminus\{\lambda_s\}\} < \lambda_s.
\end{equation}
If $k=0$, the strict inequality in \eqref{intro1.3} may fail. A more general
version of the above result is stated in Theorem~\ref{thm:1.1} of this paper
and Theorem~\ref{thm:1.1} is a special case of results in \cite{E}.  The
method of proof involves ideas from the theory of positive linear operators,
particularly generalizations of the Kre{\u\i}n-Rutman theorem to noncompact
linear operators; see \cite{Krein-Rutman}, \cite{Bonsall},
\cite{Schaefer-Wolff}, \cite{A}, \cite{L}, \cite{E}, and
\cite{Mallet-Paret-Nussbaum}. We do not use the thermodynamic formalism (see
\cite{Ruelle}) and often our operators cannot be studied in Banach spaces of
analytic functions.

The linear operators which are relevant for the computation of Hausdorff
dimension comprise a small subset of the transfer operators described in
\eqref{intro1.2}, but the analysis problem which we shall consider here can be
described in the generality of \eqref{intro1.2} and is of interest in this
more general context. We want to find rigorous methods to estimate
$r(L_{s,k})$ accurately and then use these methods to estimate $s_*$, where,
in our applications, $s_*$ will be the unique number $s \ge 0$ such that
$r(L_{s,k})=1$.  Under further assumptions, we shall see that $s_*$ equals
$\dim_H(C)$, the Hausdorff dimension of the invariant set associated to the
IFS.  This observation about Hausdorff dimension has been made, in varying
degrees of generality by many authors. See, for example, \cite{Bumby1},
\cite{Bumby2}, \cite{Bowen}, \cite{Cusick1}, \cite{Cusick2}, \cite{Falconer},
\cite{Good}, \cite{Hensley1}, \cite{Hensley2}, \cite{Hensley3},
\cite{J}, \cite{Jenkinson}, \cite{Jenkinson-Pollicott},
\cite{MR1902887}, \cite{H}, \cite{Mauldin-Urbanski}, \cite{N-P-L},
\cite{Ruelle}, \cite{Ruelle2}, \cite{Rugh}, and \cite{Schief}.

We assume in this paper that $H$ is a bounded, open mildly regular subset of
$\R^2 = \C$ and that $\theta_{\be}$, $\be \in \B$, are analytic or conjugate
analytic contraction maps, defined on an open neighborhood of $\bar H$ and
satisfying $\theta_{\be}(H) \subset H$.  We define $D \theta_{\be}(z) =
\lim_{h \rightarrow 0} |[\theta_{\be}(z+h)- \theta_{\be}(z)]/h|$, where $h \in
\C$ in the limit, and we assume that $D \theta_{\be}(z)\neq 0$ for $z \in \bar
H$. In this case, $L_{s,k}$ is defined by \eqref{intro1.2}, with $x$ replaced
by $z$, and $g_{\be}(z)= D \theta_{\be}(z)$.  It is then possible to obtain
explicit upper and lower bounds for $D_1^p v_s(x_1,x_2))/v_s(x_1,x_2)$ and
$D_2^p v_s(x_1,x_2))/v_s(x_1,x_2)$, where $D_1 = \partial/\partial x_1$ and
$D_2 = \partial/\partial x_2$. However, for simplicity we restrict ourselves
to the choice $\theta_{\be}(z) = (z + \be)^{-1}$, where $\be \in \C$ and
$\Re(\be) > 0$. In this case we obtain in Section~\ref{sec:mobius} explicit
upper and lower bounds for $D_k^p v_s(x_1,x_2))/v_s(x_1,x_2)$ for $1 \le p \le
4$, $1 \le k \le 2$, and $x_1 >0$.  In both the one and two dimensional cases,
these estimates play a crucial role in allowing us to obtain rigorous upper
and lower bounds for the Hausdorff dimension. Of course, obtaining these
estimates adds to the length of \cite{hdcomp1} and this paper.  However, aside
from their intrinsic interest, we believe these results will prove useful in
other contexts, e.g., in treating generalizations of the {\it Texan
  conjecture} (see \cite{K-Z} and \cite{Jenkinson}).

The basic idea of our numerical scheme is to cover $\bar H$ by nonoverlapping
squares of side $h$.  We remark that our collection of squares need not be a
{\it Markov partition} for our IFS; compare \cite{McMullen}. We then
approximate the strictly positive, $C^2$ eigenfunction $v_s$ by a continuous
piecewise bilinear function.  Using the explicit bounds on the unmixed
derivatives of $v_s$ of order $2$, we are then able to associate to the
operator $L_{s,k}$, square matrices $A_s$ and $B_s$, which have nonnegative
entries and also have the property that $r(A_s) \le \lambda_s \le r(B_s)$. A
key role here is played by an elementary fact (see Lemma~\ref{lem:nonneg} in
Section~\ref{sec:prelim}) which is not as well known as it should be and in
the matrix case reduces to the following observation:  If $M$ is a nonnegative
matrix and $v$ is a strictly positive vector and $M v \le \lambda v$,
(coordinate-wise), then $r(M) \le \lambda$.  Analogously, $r(M) \ge \lambda$
if $M v \ge \lambda v$.

If $s_*$ denotes the unique value of $s$ such that $r(L_{s_*}) = \lambda_{s_*}
= 1$, so that $s_*$ is the Hausdorff dimension of the invariant set for the
IFS under study, we proceed as follows.  If we can find a number $s_1$ such
that $r(B_{s_1}) \le 1$, then, since the map $s \mapsto \lambda_s$ is
decreasing, $\lambda_{s_1} \le r(B_{s_1}) \le 1$, and we can conclude that
$s_* \le s_1$.  Analogously, if we can find a number $s_2$ such that
$r(A_{s_2}) \ge 1$, then $\lambda_{s_2} \ge r(A_{s_2}) \ge 1$, and we can
conclude that $s_* \ge s_2$.  By choosing the mesh size for our approximating
piecewise polynomials to be sufficiently small, we can make $s_1-s_2$ small,
providing a good estimate for $s_*$.  For a given $s$, $r(A_s)$ and $r(B_s)$
are easily found by variants of the power method for eigenvalues, since the
largest eigenvalue of $A_s$ (respectively, of $B_s$) has multiplicity one and
is the only eigenvalue of its modulus.  When the IFS is infinite, the
procedure is somewhat more complicated, and we include the necessary theory to
deal with this case.

This new approach was illustrated in \cite{hdcomp1}  by first
considering the computation of the Hausdorff
dimension of invariant sets in $[0,1]$ arising from classical continued
fraction expansions.  In this much studied case, one defines $\theta_m(x) =
1/(x+m)$, for $m$ a positive integer and $x \in [0,1]$; and for a subset $\B
\subset \N$, one considers the IFS $\{\theta_m: m \in \B\}$ and seeks
estimates on the Hausdorff dimension of the invariant set $C =C(\B)$ for this
IFS.  This problem has previously been considered by many authors. See
\cite{Bourgain-Kontorovich}, \cite{Bumby1}, \cite{Bumby2}, \cite{Good},
\cite{Hensley1}, \cite{Hensley2}, \cite{Hensley3}, \cite{Jenkinson},
\cite{Jenkinson-Pollicott}, and \cite{Heinemann-Urbanski}.  In this case,
\eqref{intro1.2} becomes
\begin{equation*}
(L_{s,k}v)(x) = \sum_{m \in \B} \Big(\frac{1}{x+m}\Big)^{2s} 
v\Big(\frac{1}{x+m}\Big), \qquad 0 \le x \le 1,
\end{equation*}
and one seeks a value $s \ge 0$ for which $\lambda_s:= r(L_{s,k}) =1$.

In Section~\ref{sec:2dexp}, we consider the computation of the Hausdorff
dimension of some invariant sets arising from complex continued fractions.
Suppose that $\B$ is a subset of $I_1 := \{m+ni : m \in \N, n\in \Z\}$, and
for each $b \in \B$, define $\theta_{\be}(z) = (z+\be)^{-1}$. Note that
$\theta_{\be}$ maps $\bar G = \{z \in \C : |z-1/2| \le 1/2\}$ into itself.  We
are interested in the Hausdorff dimension of the invariant set $C = C(\B)$ for
the IFS $\{\theta_{\be} : \be \in \B\}$.  This is a two dimensional problem
and we allow the possibility that $\B$ is infinite. In general (contrast work
in \cite{Jenkinson-Pollicott} and \cite{Jenkinson}), it does not seem possible
in this case to replace $L_{s,k}$, $k \ge 2$, by an operator $\Lambda_s$
acting on a Banach space of analytic functions of one complex variable and
satisfying $r(\Lambda_s) = r(L_{s,k})$.  Instead, we work in $C^2(\bar G)$ and
apply our methods to obtain rigorous upper and lower bounds for the Hausdorff
dimension $\dim_H(C(\B))$ for several examples. The case $\B = I_1$ has been
of particular interest and is one motivation for this paper.  In
\cite{Gardner-Mauldin}, Gardner and Mauldin proved that $d:= \dim_H(C(I_1))
<2$. In Theorem 6.6 of \cite{Mauldin-Urbanski}, Mauldin and Urbanski proved
that $1.2484 \le d \le 1.885$, and in \cite{Priyadarshi}, Priyadarshi proved
that $d \ge 1.78$.  In Section~\ref{subsec:method}, we show (modulo roundoff
errors in the calculation) that $1.85574 \le d \le 1.85589$. We believe (see
Remark~\ref{rem:interval} in Section~\ref{sec:2dexp}) that this estimate can
be made rigorous by using interval arithmetic along with high order precision,
although since we consider this paper to be a feasibility study, we have not
done this.

In the case when the eigenfunctions $v_s$ have additional smoothness, it is
natural to approximate $v_s(\cdot)$ by piecewise tensor product polynomials of
higher degree.  In this situation, the corresponding matrices $A_s$ and $B_s$
may no longer have all nonnegative entries and so the arguments of this paper
are no longer directly applicable.  However, as demonstrated in
Table~\ref{tb:t5} and Table~\ref{tb:t6}, this approach gives much improved
estimates for the value of $s$ for which $r(L_s)=1$.  It is our intent to
develop an extension of our theory to make these into rigorous bounds.

It is also worth comparing the approach used in our paper with that of
McMullen \cite{McMullen}. Superficially the methods seem different, but there
are underlying connections.  We exploit the existence of a $C^k$, strictly
positive eigenfunction $v_s$ of \eqref{intro1.2} with eigenvalue $\lambda_s$
equal to the spectral radius of $L_{s,k}$; and we observe that explicit bounds
on derivatives of $v_s$ can be exploited to prove convergence rates on
numerical approximation schemes which approximate $\lambda_s$.  McMullen does
not explicitly mention the operator $L_{s,k}$ or the analogue of $L_{s,k}$ for
graph directed iterated function systems, and he does not use $C^k$, strictly
positive eigenfunctions of equations like \eqref{intro1.2} or obtain bounds on
partial derivatives of such positive eigenfunctions.  Instead, he exploits
finite positive measures $\mu$ which are called ``$\mathcal{F}-$invariant
densities of dimension $\delta$.'' If $s_*$ is a value of $s$ for which the
above eigenvalue $\lambda_s =1$, then in our context the measure $\mu$ is an
eigenfunction of the Banach space adjoint $(L_{s_*,0})^*$ with eigenvalue $1$,
and our $s_*$ corresponds to $\delta$ above.  Standard arguments using
weak$^*$ compactness, the Schauder-Tychonoff fixed point theorem, and the
Riesz representation theorem imply the existence of a regular, finite,
positive, complete measure $\mu$, defined on a $\sigma$-algebra containing all
Borel subsets of the underlying space $\bar H$ and such that $(L_{s_*,0})^*
\mu = \mu$ and $\int v_{s_*} \, d \mu =1$.

McMullen also uses refinements of {\it Markov partitions,} while our
partitions, both here and in \cite{hdcomp1}, need not be Markov. However, in
the end, both approaches generate (different) $n \times n$ nonnegative
matrices $M_s$, parametrized by a parameter $s$ and both methods use the
spectral radius of $M_s$ to approximate the desired Hausdorff dimension $s_*$.
McMullen's matrices are obtained by approximating certain nonconstant
functions defined on a refinement of the original Markov partition by
piecewise constant functions defined with respect to this refinement.  We
approximate by bilinear functions on each subset in our partition.  As we show
below, by exploiting estimates on higher derivatives of $v_s(\cdot)$, our
methods give explicit upper and lower bounds for $s_*$ and more rapid
convergence to $s_*$ than one obtains using piecewise constant approximations.

The square matrices $A_s$ and $B_s$ mentioned above and described in more
detail later in the paper have nonnegative entries and satisfy $r(A_s) \le
\lambda_s \le r(B_s)$.  To apply standard numerical methods, it is useful to
know that all eigenvalues $\mu \neq r(A_s)$ of $A_s$ satisfy $|\mu| < r(A_s)$
and that $r(A_s)$ has algebraic multiplicity one and that corresponding
results hold for $r(B_s)$.  Such results were proved in Section 7 of
\cite{hdcomp1} in the one dimensional case when the mesh size, $h$, is
sufficiently small, and a similar argument can be used in the two dimensional
case under study here.  Note that this result does not follow from the
standard theory of nonnegative matrices, since $A_s$ and $B_s$ typically have
zero columns and are not primitive. As in \cite{hdcomp1}, we can also prove
that $r(A_s) \le r(B_s) \le (1 + C_1 h^2) r(A_s)$, where the constant $C_1$
can be explicitly estimated.  In a manner exactly analogous to that used in
\cite{hdcomp1}, it can be proved (see Theorem 7.1) that the map $s \mapsto
\lambda_s$ is log convex and strictly decreasing; and this same result holds
for $s \mapsto r(M_s)$, where $M_s$ is a naturally defined matrix such that
$A_s \le M_s \le B_s$.  This idea is exploited in our computer code in the
following way. Recall that if we can find a number $s_1$ such that $r(B_{s_1})
\le 1$, then, since the map $s \mapsto \lambda_s$ is decreasing,
$\lambda_{s_1} \le r(B_{s_1}) \le 1$, and we can conclude that $s_* \le
s_1$. To obtain the best bound, we seek a value $s_1$ such that $r(B_{s_1})$
is as close as possible to $1$, while still remaining $\le 1$.  This is done
by a slight modification of the secant method applied to finding a zero of the
function $\log[r(B_{s_1})]$. A similar approach is used with $A_s$ to find a
lower bound for $s_*$.

A summary of the paper is as follows.  In Section~\ref{sec:prelim}, we recall
the definition of Hausdorff dimension and present some mathematical
preliminaries. In Section \ref{sec:2dexp}, we present the details of our
approximation scheme for Hausdorff dimension, explain the crucial role played
by estimates on unmixed partial derivatives of order $\le 2$ of $v_s$, and
give the aforementioned estimates for Hausdorff dimension.  We emphasize that
this is a feasibility study. We have limited the accuracy of our
approximations to what is easily found using the standard precision of {\it
  Matlab} and have run only a limited number of examples, using mesh sizes
that allow the programs to run fairly quickly. In addition, we have not
attempted to exploit the special features of our problems, such as the fact
that our matrices are sparse.  Thus, it is clear that one could write a more
efficient code that would also speed up the computations.  However, the {\it
  Matlab} programs we have developed are available on the web at {\tt
  www.math.rutgers.edu/\char'176falk/hausdorff/codes.html}, and we hope other
researchers will run other examples of interest to them.

The theory underlying the work in Section~\ref{sec:2dexp} is presented in
Sections~\ref{sec:exist}--\ref{sec:logconvex}.  In Section~\ref{sec:exist} we
describe some results concerning existence of $C^m$ positive eigenfunctions
for a class of positive (in the sense of order-preserving) linear operators.
We remark that Theorem~\ref{thm:1.1} in Section~\ref{sec:exist} was only
proved in \cite{E} for finite IFS's.  As a result, some care is needed in
dealing with infinite IFS's.  In Section~\ref{sec:mobius}, we derive explicit
bounds on the partial derivatives of eigenfunctions of operators in which the
mappings $\theta_{\be}$ are given by M\"obius transformations which map a
given bounded open subset $H$ of $\C:= \R^2$ into $H$.  We use this
information in Theorems~\ref{thm:thm5.10}-\ref{thm:thm5.13} to obtain results
about the case of infinite IFS's which are adequate for our immediate
purposes. In Section~\ref{sec:compute-sr}, we verify some spectral properties
of the approximating matrices which justify standard numerical algorithms for
computing their spectral radii. Finally, in Section~\ref{sec:logconvex}, we
discuss the log convexity of the spectral radius $r(L_s)$, which we exploit in
our numerical approximation scheme.

\section{Preliminaries}
\label{sec:prelim}
We recall the definition of the Hausdorff dimension, $\dim_H(K)$, of
a subset $K \subset \R^N$.  To do so, we first define for a given $s \ge0$
and each set $K \subset \R^N$, 
\begin{equation*}
H_{\delta}^s(K) = \inf\{\sum_i |U_i|^s: \{U_i\} \text{ is a } \delta \text{ cover
of } K\},
\end{equation*}
where $|U|$ denotes the diameter of $U$ and a countable collection $\{U_i\}$
of subsets of $\R^N$ is a $\delta$-cover of $K \subset \R^N$ if
$K \subset \cup_i U_i$ and $0 < |U_i| < \delta$ for all $i$.  We then define
the $s$-dimensional Hausdorff measure
\begin{equation*}
H^s(K) = \lim_{\delta \rightarrow 0+} H_{\delta}^s(K).
\end{equation*}
Finally, we define the Hausdorff dimension of $K$, $\dim_H(K)$, as
\begin{equation*}
\dim_H(K) = \inf\{s: H^s(K) =0\}.
\end{equation*}

We now state the main result connecting Hausdorff dimension to the spectral
radius of the map defined by \eqref{intro1.2}.  To do so, we first define the
concept of an {\it infinitesimal similitude}.  Let $(S,d)$ be a bounded,
complete, perfect metric space. If $\theta:S \to S$, then $\theta$ is an
infinitesimal similitude at $t \in S$ if for any sequences $(s_k)_k$ and
$(t_k)_k$ with $s_k \neq t_k$ for $k \ge 1$ and $s_k \rightarrow t$, $t_k
\rightarrow t$, the limit
\begin{equation*}
\lim_{k \rightarrow \infty} \frac{d(\theta(s_k), \theta(t_k)}{d(s_k,t_k)}
=: (D \theta)(t)
\end{equation*}
exists and is independent of the particular sequences 
$(s_k)_k$ and $(t_k)_k$.  Furthermore, $\theta$ is an infinitesimal similitude
on $S$ if $\theta$ is an infinitesimal similitude at $t$ for all $t \in S$.

This concept generalizes the concept of affine linear similitudes, which
are affine linear contraction maps $\theta$ satisfying
for all $x,y \in \R^n$
\begin{equation*}
d(\theta(x), \theta(y)) = c d(x,y), \quad c < 1.
\end{equation*}
In particular, the examples discussed in \cite{hdcomp1}, such as maps
of the form $\theta(x) = 1/(x+m)$, with $m$ a positive integer,
are infinitesimal similitudes. More generally, if $S$ is a compact subset
of $\R^1$ and $\theta:S \to S$ extends to a $C^1$ map defined on an
open neighborhood of $S$ in $\R^1$, then $\theta$ is an infinitesimal
similitude. If $S$ is a compact subset of $\R^2:=\C$ and  $\theta:S \to S$
extends to an analytic or conjugate analytic map defined on an open
neighborhood of $S$ in $\C$, $\theta$ is an infinitesimal similitude.

\begin{thm} (Theorem 1.2 of \cite{N-P-L}.)
Let $\theta_i:S \to S$ for $1 \le i \le N$ be infinitesimal similitudes
and assume that the map $t \mapsto (D\theta_i)(t)$ is a strictly positive
H\"older continuous function on $S$.  Assume that $\theta_i$ is a Lipschitz
map with Lipschitz constant $c_i \le c <1$ and let
$C$ denote the unique, compact, nonempty invariant set such that
\begin{equation*}
C = \cup_{i=1}^N \theta_i(C).
\end{equation*}
Further, assume that $\theta_i$ satisfy
\begin{equation*}
\theta_i(C) \cap \theta_{j}(C) = \emptyset, \text{ for } 1 \le i,j \le N.
\ i \neq j
\end{equation*}
and are one-to-one on $C$.  Then the Hausdorff dimension of $C$ is
given by the unique $\sigma_0$ such that $r(L_{\sigma_0}) = 1$,
where $L_s:C(S) \to C(S)$ is defined for $s \ge 0$ by
\begin{equation*}
(L_sf)(t) = \sum_{i=1}^N [D \theta_i(t)]^s f(\theta_i(t)).
\end{equation*}
\end{thm}

The following lemma is a well-known result, but we sketch the proof because
the lemma with play a crucial role in some of our later arguments.
\begin{lem}
\label{lem:nonneg}
Let $Q$ be a compact Hausdorff space, $X= C_{\R}(Q)$, the Banach space of
continuous, real-valued functions $f: Q \to \R$ in the $\sup$ norm,
\begin{equation*}
K = \{f \in X: f(t) \ge 0 \ \forall t \in Q\}, \text{ and }
\interior(K) = \{f \in X: f(t) > 0 \ \forall t \in Q\}.
\end{equation*}
If $f,g \in X$, write $f \le g$ if $g-f \in K$. Let $L:X \to X$ be a bounded
linear map such that $L(K) \subset K$ and write $r(L):=
\lim_{n \rightarrow \infty} \|L^n\|^{1/n}$, the spectral radius of $L$. If
there
exists $w \in \interior(K)$ such that $Lw \le \beta w$ for some $\beta
\in \R$, then $r(L) \le \beta$. If there exists $v \in K \setminus\{0\}$
such that $L v \ge \alpha v$ for some $\alpha \in \R$, then $r(L) \ge \alpha$.
\end{lem}
\begin{proof}
  Define $u \in K$ by $u(t) =1 \ \forall t \in Q$. If $f \in X$ and $\|f\| \le
  1$, then $- u \le f \le u$, so $-L^k u \le L^k f \le L^ku$. It follows that
  $\|L^k f\| \le \|L^k u\|$ and this implies $\|L^k\| = \|L^k u\|$ and $r(L) =
  \lim_{k \rightarrow \infty} \|L^k\|^{1/k} = \lim_{k \rightarrow \infty}
  \|L^k u\|^{1/k}$.

If $w \in \interior(K)$, there exist positive constants $c$ and $d$ such
that $c w \le u \le dw$, so, for all positive integers $k$,
\begin{equation*}
c L^k w \le L^k u \le d  L^k w \text{ and } 
c \|L^k w\| \le \|L^k u\| \le d  \|L^k w\|.
\end{equation*}
Taking $k$th roots and letting $k \rightarrow \infty$, we obtain
$r(L) = \lim_{k \rightarrow \infty} \|L^k w\|^{1/k}$. However, if
$L w \le \beta w$, $L^k w \le \beta^k w$, so
$r(L) \le  \lim_{k \rightarrow \infty} \|\beta^k w\|^{1/k} = \beta$.
If $L v \ge \alpha v$ for some $v \in K \setminus\{0\}$, then
 $L^k v \ge \alpha^k v$ for all positive integers $k$ and
$\|L^k\| \|v\| \ge \alpha^k \|v\|$. Taking $k$th roots and letting
$k \rightarrow \infty$, we find that $r(L) \ge \alpha$.
\end{proof}

Note that if we take $Q = \{1, 2, \ldots, N\}$ and identify $C_{\R}(Q)$
with column vectors in $\R^N$, Lemma~\ref{lem:nonneg} gives results
concerning $r(L)$, where $L: \R^N \to \R^N$ is an $N \times N$ matrix with
nonnegative entries, or, more abstractly, a linear map which takes the cone
of vectors $x$ with nonnegative entries into itself.

Lemma~\ref{lem:nonneg} is a special case of much more general results
concerning order-preserving, homogeneous cone mappings: see \cite{X} and also
Lemma 2.2 in \cite{C} and Theorem 2.2 in \cite{B}.  In the important special
case that $L$ is given by an $N \times N$ matrix with non-negative entries,
Lemma~\ref{lem:nonneg} can also be derived from standard results in \cite{D}
concerning nonnegative matrices.  A simple proof in the matrix case we
consider here can also be found in Lemma 2.2 in \cite{hdcomp1}.

Our next lemma is also a well-known result. Because it follows easily from
Lemma~\ref{lem:nonneg}, we leave the proof to the reader.

\begin{lem}
\label{cor-nonneg}
Let notation be as in Lemma~\ref{lem:nonneg}. Suppose that $L_j:X \to X$,
$j=1,2$, are bounded linear maps such that $L_j(K) \subset K$ and $L_1(f) \le
L_2(f)$ for all $f \in K$.  Then it follows that $r(L_1) \le r(L_2)$.  If
there exists $v \in \interior(K)$ with $L v = \lambda v$, then $r(L) =
\lambda$.
\end{lem}

\section{Iterated Function Systems Associated to Complex
Continued Fractions}
\label{sec:2dexp}

\subsection{The problems}
\label{subsec:problem}

Throughout this section we shall always write $D:= \{(x,y) \in \R^2: (x-1/2)^2
+ y^2 \le 1/4\}$ and $U$ will always denote a bounded, {\it mildly regular}
open subset of $\R^2$ such that $\interior(D) \subset U$ and $x >0$ for all
$(x,y) \in U$, while $H$ will denote $\{(x,y) \in U: y > 0\}$.  By writing
$z = x+ \irm y$, we can consider $D$, $H$, and $U$ as subsets of the complex
plane. If $S \subset \R^2$, we shall use the identification of $\R^2$ with
$\C$ and say that $S$ is symmetric under conjugation if $S=\{\bar z: z \in
S\}$, where $\bar z$ denotes the complex conjugate of $z$.

In this section, $\B$ will always denote a finite or countable
infinite subset of $\{w \in \C:= \R^2: \Re(w) \ge 1\}$, and for $b \in
\B$, $\theta_b$ will denote the M\"obius transform $z \mapsto 1/(z+b):
= \theta_b(z)$.  If $G:=\{z \in \C: \Re(z) \ge 0\}$, the reader can
check that for all $b \in \B$, $\theta_b(G) \subset D \setminus
\{0\}$; and if $b, c \in \B$ satisfy $\Re(b) \ge \gamma \ge 1$ and
$\Re(c) \ge \gamma \ge 1$, then $\theta_b \circ \theta_c: G \mapsto D
\setminus \{0\}$ is a Lipschitz map (with respect to the Euclidean
metric) with Lipschitz constant $\Lip(\theta_b \circ \theta_c) \le
(\gamma^2 +1)^{-2}$ (see Lemma~\ref{lem:4.1} below). We shall always
write $I_1 : = \{\be = m + n\irm: m \in \N, n \in \Z\}$ and the case
that $\B \subset I_1$ will be of particular interest.

We shall denote by $C_{\C}(\bar U)$ (respectively, $C_{\R}(\bar U)$) the 
Banach space of continuous maps
$f:\bar U \to \C$ (respectively, $f:\bar U \to \R$) with
$\|f\| = \max\{|f(z)|: z \in \bar U\}$. (Note that $\bar U$ will always
denote the closure of $U$ and {\it not} the image of $U$ under complex
conjugation.) If $\B$ is a finite set
and $s > 0$, one can define a bounded, complex linear map
$L_s: C_{\C}(\bar U) \to C_{\C}(\bar U)$ by
\begin{equation}
\label{Lsdefz}
(L_sf)(z) = \sum_{\be \in \B} \Big|\frac{d}{dz} \theta_{\be}(z)\Big|^s
f(\theta_{\be}(z)) = \sum_{\be \in \B} \frac{f(\theta_{\be}(z))}{|z+\be|^{2s}}.
\end{equation}
Equation \eqref{Lsdefz} also defines a bounded, real linear map
of $C_{\R}(\bar U) \to C_{\R}(\bar U)$, which (abusing notation) we shall
also denote by $L_s$. We shall denote by $\sigma(L_s)$ the spectrum
of $L_s: C_{\C}(\bar U) \to C_{\C}(\bar U)$.

If $\B$ is infinite, one can prove (see Section 5 of \cite{A} and
\cite{N-P-L}) that if, for some $s >0$, the infinite series $\sum_{\be \in \B}
[1/|\be|^{2s}]$ converges, then $\sum_{\be \in \B} [1/|z+\be|^{2s}]$ converges
for all $z \in \bar U$ and gives a continuous function on $\bar U$.  It then
follows with the aid of Dini's theorem that $L_s$ given by \eqref{Lsdefz}
defines a bounded linear map of $C_{\C}(\bar U)$ to itself.
If we define $\tau = \tau(\B):= \inf\{s >0 : \sum_{\be \in \B}
[1/|\be|^{2s}] < \infty\}$ (where we allow $\tau(\B) = \infty$), it follows
from the above remarks that for all $s > \tau(\B)$, $L_s$ 
gives a bounded linear map of $C_{\C}(\bar U)$ to itself.  If
$s= \tau$, it may or may not happen that $\sum_{\be \in \B} [1/|\be|^{2s}]
< \infty$. In any event, we shall show that if
$s > 1$, $\sum_{\be \in \B} [1/|\be|^{2s}] < \infty$.

Our goal in the section is to describe how to obtain rigorous upper and lower
bounds for $r(L_s)$, the spectral radius of the operator $L_s$ in
\eqref{Lsdefz}, and then to indicate how such bounds enable us to rigorously
estimate the Hausdorff dimension of some interesting sets.
To avoid interrupting the narrative flow, we first list some results which we
shall need, but whose proofs will be deferred to Sections~\ref{sec:exist}
and ~\ref{sec:mobius}. If $\alpha \ge 0$, $R >0$, and $\B$ is as before, we
define
\begin{equation*}
\B_R = \{\be \in \R : |\be| \le R \} \qquad \text{and}
\qquad   \B_R^{\prime} = \{\be \in \R : |\be| > R \}.
\end{equation*}
If $\B$ is finite, we shall usually take $R \ge \sup\{|\be|: \be \in \B\}$,
so $\B_R=\B$.  We define $L_{s,R,\alpha}: C_{\C}(\bar U) \to C_{\C}(\bar
U)$ by
\begin{equation}
\label{Lsdefza}
(L_{s,R, \alpha}f)(z) = \sum_{\be \in \B_R} 
\frac{f(\theta_{\be}(z))}{|z+\be|^{2s}} + \alpha f(0).
\end{equation}

\begin{thm}
\label{thm:thm3.1}
Assume that $\B$ is finite and $\Re(\be) \ge \gamma \ge 1$ for all $\be \in
\B$. For each $s \ge 0$, there exists a unique (to within scalar multiples)
strictly positive continuous eigenfunction $w_s \in C_{\R}(\bar U)$ with
positive eigenvalue $r(L_{s,R, \alpha})$ defined by
$r(L_{s,R, \alpha}):= \lim_{k \rightarrow   \infty} \|L_{s,R, \alpha}^k\|^{1/k}$. 
(Of course $w_s$ also depends on $\alpha$ and $R$, but we view $\alpha$ and $R$
as fixed and omit the dependence in our notation.) If $\B$ and $U$ are
symmetric under conjugation, then $w_s(\bar z) = w_s(z)$ for all $z \in \bar
U$. In general, identifying $(x,y) \in \R^2$ with $x+ iy \in \C$,
$w_s(x,y)$ is $C^{\infty}$ on $\bar U$ and the following estimates hold.
\begin{align}
\label{wrelation}
w_s(z_0) &\le w_s(z_1) \exp[(\sqrt{5}s/\gamma)|z_1-z_0|], \quad z_0,z_1 \in
\bar U,
\\
\label{wrelation2}
w_s(x_1,y) &\ge w_s(x_2,y) \ge w_s(x_1,y)\exp[(-2s/\gamma)(x_2-x_1)], 
\\
\notag
& \qquad 0 \le x_1 \le x_2, \quad (x_1,y), (x_2,y) \in \bar U,
\\
\label{wrelation3}
w_s(x,y_1) &\le w_s(x,y_2) \exp[(s/\gamma)|y_1-y_2|], 
\quad (x,y_1), (x,y_2) \in \bar U,
\\
\label{Dxxbound}
-\frac{s}{4\gamma^2(s+1)} w_s(x,y) &\le  D_{xx} w_s(x,y) \le 
\frac{2s(2s+1)}{\gamma^2} w_s(x,y),
\\
\label{Dyybound}
-\frac{2s}{\gamma^2} w_s(x,y) &\le  D_{yy} w_s(x,y) 
\le \frac{2s(2s+1)}{4\gamma^2} w_s(x,y).
\end{align}
\end{thm}

\begin{thm}
\label{thm:thm3.2}
Assume that $\B$ is infinite and that $s >0$ satisfies $\sum_{\be \in \B}
[1/|\be|^{2s}] < \infty$.  Then $L_s$ has a unique (to within
scalar multiples) strictly positive eigenfunction $v_s \in C_{\R}(\bar U)$
with positive eigenvalue $r(L_{s})$. This eigenfunction is Lipschitz and
satisfies \eqref{wrelation}, \eqref{wrelation2}, and \eqref{wrelation3}.  If
$\B$ and $U$ are symmetric under conjugation, then $v_s(\bar z) = v_s(z)$ for
all $z \in U$.
\end{thm}

\begin{thm}
\label{thm:thm3.3}
Let assumptions and notation be as in Theorem~\ref{thm:thm3.2} and assume that
$R >2$. Then there exist (see Theorems~\ref{thm:thm5.12} and
\ref{thm:thm5.13}) real numbers $\eta_{s,R} \ge 0$ and $\delta_{s,R} >0$ such
that
\begin{equation*}
\eta_{s,R} v_s(0) \le \sum_{\be \in \B, |\be| >R} \frac{v_s(\theta_{\be}(z))}
{|z+b|^{2s}} \le \delta_{s,R} v_s(0).
\end{equation*}
If $\B = I_1$ or $\B = I_2:= \{m+n \irm : m \in \N, n \in \Z, n <0\}$
and $s >1$, explicit estimates for $\eta_{s,R}$ and $\delta_{s,R}$ are given
in Theorems~\ref{thm:thm5.12} and \ref{thm:thm5.13}. If $\alpha = \delta_{s,R}$,
\begin{equation}
\label{specrel1}
r(L_s) \le r(L_{s,R, \alpha});
\end{equation}
and if $\alpha = \eta_{s,R}$,
\begin{equation}
\label{specrel2}
r(L_{s,R, \alpha}) \le r(L_s). 
\end{equation}
\end{thm}

If $\B$ is finite, we shall usually assume that $|\be| \le R$ for all
$\be \in \B$ and take $\alpha =0$.  If $\B$ is infinite, we take $R$ large
and use \eqref{specrel1} and \eqref{specrel2} to estimate $r(L_s)$.  In all
cases our problem reduces to finding a procedure which gives rigorous
upper and lower bounds for operators $L_{s,R,\alpha}$, where
$\alpha = \delta_{s,R}$ or $\alpha = \eta_{s,R}$, or $\alpha =0$.

If $\B$ and $U$ are symmetric under conjugation, let $H$ be as defined at the
beginning of this section and let $\bar H$ denote the closure of $H$.  Let $Y=
\{f \in C_{\C}(\bar U) : f(z) = f(\bar z), z \in \bar U\}$, so $Y$ is a
complex Banach space, and one can check that $Y$ is linearly isometric to
$C_{\C}(\bar H)$ by $f \in Y \mapsto f|_{\bar H} \in C_{\C}(\bar H)$ and $g
\in C_{\C}(\bar H) \mapsto \tilde g \in Y$, where $\tilde g(z) = g(z)$ if $z
\in \bar H$ and $\tilde g(z) = g(\bar z)$ if $z \in \bar U$ and $z \notin \bar
H$. In the notation of Theorem~\ref{thm:thm3.2}, $w_s \in Y$, and the reader
can check that $L_{s,R,\alpha}$ maps $Y$ into $Y$, Equivalently,
$L_{s,R,\alpha}$ can be viewed as a bounded linear map of $C_{\C}(\bar H)$ to
$C_{\C}(\bar H)$ by defining $f(1/(z+b)) = f(1/(\bar z + \bar b))$ if
$\Im(z+b) \ge 0$ and $f(1/(z+b)) = f(1/(z+b))$ if $\Im(z+b) \le 0$.  This
observation will simplify the numerical analysis in later examples.

If $\Im(\be) \le -1$ for all $\be \in \B$ (but without the assumption that
$\B$ and $U$ are symmetric under conjugation) and if $\Im(z) \le 1$ for
all $z \in \bar U$, one can easily verify that $\theta_{\be}(z) \in \bar H$
for all $\be \in \B$ and $z \in \bar U$.  Thus, again in this case one
can consider $L_{s,R,\alpha}$ as a map of $C_{\C}(\bar H)$ to itself,
which again will simplify the numerical analysis.

We now briefly discuss the connection of
Theorems~\ref{thm:thm3.1}-\ref{thm:thm3.3} to the problem of computing the
Hausdorff dimension of certain sets.


If $\B \subset I_1$, let $\B_{\infty} = \{\w = (\be_1, \ldots, \be_k, \ldots)
: \be_j \in \B \ \forall j \ge 1\}$.  Given $z \in D$ and
$\w  =(\be_1, \ldots, \be_k, \ldots) \in \B_{\infty}$, one can prove that
$\lim_{k \rightarrow \infty}(\theta_{\be_1} \circ \theta_{\be_2} \circ \cdots \circ 
\theta_{\be_k})(z) := \pi(\w) \in D$ exists and is independent of $z$.
Define $C = \{\pi(\w) : \w \in \B_{\infty}\}$.  It is not hard to
prove that $C = \cup_{\be \in \B} \theta_{\be}(C)$.  In general $C$ is
not compact, but if $\B$ is finite, $C$ is compact and is the unique
compact, nonempty set $C$ such that $C = \cup_{\be \in \B} \theta_{\be}(C)$.
We shall call $C$ the invariant set associated to $\B$.

\begin{thm}
  \label{thm:thm3.4} (See Section 5 of \cite{N-P-L}) Let $\B$ be a subset of
$I_1$, let $L_s: C_{\R}(\bar U) \to C_{\R}(\bar U)$ be defined by
\eqref{Lsdefz} for $s > \tau(\B)$, and let $C$ be the invariant set
associated to $\B$. The Hausdorff dimension $s_*$ of $C$ is given by $s_* =
\inf\{s >0 : r(L_s) = \lambda_s <1\}$ and $r(L_{s_*}) = 1$ if $\B$ is
finite or $L_{s_*}$ is defined.  The map $s \mapsto \lambda_s$ is a
continuous, strictly decreasing function for $s > \tau(\B)$.
\end{thm}
In all examples which we shall consider, $L_s$ is a bounded linear map of
$C_{\C}(U) \to C_{\C}(U)$ for $s = s_*$ and $r(L_{s_*}) = 1$.

Theorems~\ref{thm:thm3.1}--\ref{thm:thm3.4} reduce the problem of estimating
the Hausdorff dimension of the invariant set $C$ for $\B \subset I_1$ to the
problem of estimating the value of $s$ for which $r(L_s) =1$.  If $\B$ is
finite, we have to estimate $r(L_{s,R, \alpha})$ for $\alpha =0$.  If
$\B$ is infinite, Theorem~\ref{thm:thm3.3} implies that we need a lower bound
for $r(L_{s,R, \alpha})$ for $\alpha = \eta_{s,R}$ and an upper bound for
$r(L_{s,R, \alpha})$ for $\alpha = \delta_{s,R}$.


If $\B = I_1$, it was stated in \cite{Mauldin-Urbanski} that the Hausdorff
dimension of the associated invariant set $C$ is $\le 1.885$ and in \cite
{Priyadarshi}, it was shown that the Hausdorff dimension of the set $C$ is
$\ge 1.78$. We shall give much sharper estimates below.  We shall also give
estimates for the Hausdorff dimension of the associated invariant set of $\B
\subset I_1$ for some other choices of $\B$, e.g., 
\begin{align*}
&\B = I_2:= \{\be = m+ n\irm: m \in \N, -n \in \N\},
\\
&\B = I_3:= \{\be = m+ n\irm: m \in \{1,2\}, n \in \{0, \pm 1, \pm 2\}\}.
\end{align*}
This is a feasibility study, so we restrict attention to these examples,
but our approach applies to general sets $\B \subset I_1$; and in fact
invariant sets for many other {\it iterated function systems} can be
handled by similar methods.

\subsection{Numerical Method}
\label{subsec:method}
Let $N>0$ be an even integer, $h:= 1/N$, and let $D$, $U$, and $H$ be
as in Section~\ref{subsec:problem}. Define $D_{+}=\{(x,y) \in D: y \ge
0\}$.  We consider {\it mesh points} of the form $(jh,kh)$, where $j
\in \N \cup \{0\}$ and $k \in \Z$. Each mesh point $(x_j,y_k) = (jh,
kh)$ defines a closed {\it mesh square} $R_{jk}$ with vertices
$(x_j,y_k)$, $(x_{j+1},y_k)$, $(x_j,y_{k+1})$, and
$(x_{j+1},y_{k+1})$.  If $D_h$ (respectively, $D_{+,h}$) is a finite
union of mesh squares and $D_h \supset D$ (respectively $D_{+,h}
\supset D_+$), $D_h$ will be called a {\it mesh domain} for $D$
(respectively, a {\it mesh domain} for $D_+$). We could choose
$D_{+,h} = [0,1] \times [0,1/2]$, but that choice would add unknowns
we do not use. Thus we shall usually take $D_h$ (respectively,
$D_{+,h}$) to be the union of squares $R_{j,k}$ which have nonempty
intersection with the interior of $D$ (respectively, $D_{+})$. The
domain $D_+$ and a mesh domain $D_{+,h}$ are illustrated in
Figure~\ref{fg:f1}.

The mesh domains $D_h$ and $D_{+,h}$ correspond to sets $\bar U$ and $\bar H$
in Section~\ref{subsec:problem}.  If $D$ and $\B$ are symmetric under
conjugation or if $\Im(\be) \le -1$ for all $\be \in \B$, the observations in
Section~\ref{subsec:problem} show that we can restrict attention to $D_+$ and
$D_{+,h}$ instead of the full sets $D$ and $D_h$.  Indeed, this will be the
case for the invariant sets associated to $I_1$, $I_2$, and $I_3$.  We also
note that in the case $\B = I_3$, there is a smaller domain $C \subset D$,
symmetric under conjugation, such that $\theta_{\be}(C) \subset C \setminus
\{0\}$ for $\be \in \B$. Although we have not done so, we could have reduced
the size of the approximate problem by using a mesh domain $C_h$ for $C$.

\begin{figure}[htb]
\centerline{\includegraphics[width=14cm]{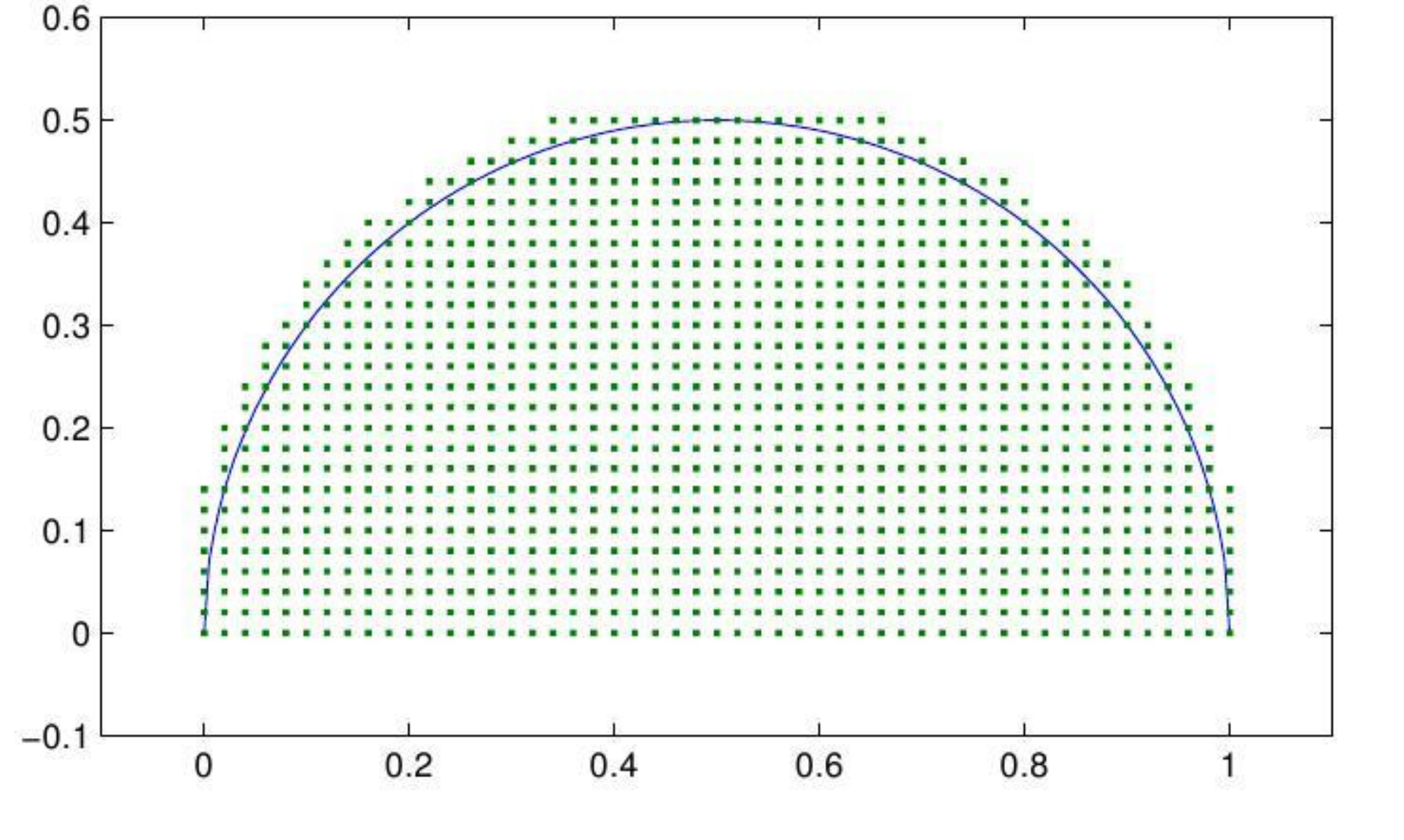}}
\caption[]{Domain $D_+$ and mesh domain $D_{+,h}$}.
\label{fg:f1}
\end{figure}

If $D_h$ is as above, we take $\bar U = D_h$ and we assume that $0 \le
x \le 1$ and $|y| <1$ for all $(x,y) \in \bar U$.  Given a set $\B
\subseteq I_1$ and $s >0$, we assume that $s > \tau(\B)$ (so $\sum_{b
  \in \B} (1/|b|^{2s}) < \infty$).  If $\B$ is finite, we assume that
$R \ge |\be|$ for all $\be \in \B$ and define $L_s: = L_{s,R, \alpha}$
with $\alpha =0$.  If $\B$ is infinite, we assume for the moment that
we have found $\eta_{s,R} \ge 0$ and $\delta_{s,R} >0$ satisfying
\eqref{specrel1} and \eqref{specrel2}. For $\alpha = \eta_{s,R}$, we
define $L_{s,R-} = L_{s,R, \alpha}$ and for $\alpha = \delta_{s,R}$,
we define $L_{s,R+} = L_{s,R, \alpha}$ (compare \eqref{Lsdefza}); we
recall that Theorem~\ref{thm:thm3.3} implies that
\begin{equation*}
r(L_{s,R-}) \le r (L_s) \le r(L_{s,R+}).
\end{equation*}

In all cases, we have an operator $L_{s,R, \alpha}$ where $\alpha \ge 0$
and $R >2$. Theorem~\ref{thm:thm3.1} implies that $L_{s,R, \alpha}$ has a
unique (to within scalar multiples) strictly positive eigenfunction $w_s$ on
$\bar U = D_h$ which has (assuming $\alpha >0$ or $\B_R \neq \emptyset$)
eigenvalue $r(L_{s,R, \alpha}) >0$.  The eigenfunction $w_s$ is
$C^{\infty}$ and satisfies \eqref{wrelation}--\eqref{Dyybound}. If $\B$ is
symmetric under conjugation, $w_s(\bar z) = w_s(z)$ for all $z \in D_h$.

We shall now describe how to find rigorous upper and lower bounds for
$r(L_{s,R, \alpha})$, where $\alpha \ge 0$ or $\B_R \neq \emptyset$.
After estimating $\eta_{s,R}$ and $\delta_{s,R}$, this will yield
rigorous upper and lower bounds for $r(L_s)$.  Our approach is to
approximate $w_s$ by a continuous, piecewise bilinear function, i.e.,
$w_s$ will be bilinear on each mesh square $R_{j,k}$ of the mesh
domain $D_h$. As noted in Section~\ref{subsec:problem}, we shall be
able to work on $D_{+,h}$ in our particular examples.


More precisely, for fixed $R$ and $\alpha$, our goal is to define nonnegative,
square matrices $A_s$ and $B_s$ such that
\begin{equation*}
r(A_s) \le r(L_s) \le r(B_s), \quad s > \tau(\B).
\end{equation*}
If $s_*$ denotes the unique value of $s$ such that $r(L_{s_*}) = \lambda_{s_*}
= 1$, then $s_*$ is the Hausdorff dimension of the invariant set associated
with $\B$.  If we can find a number $s_1$ such that $r(B_{s_1}) \le 1$, then
$r(L_{s_1}) \le r(B_{s_1}) \le 1$, and we can conclude that $s_* \le s_1$.
Analogously, if we can find a number $s_2$ such that $r(A_{s_2}) \ge 1$, then
$r(L_{s_2}) \ge r(A_{s_2}) \ge 1$, and we can conclude that $s_* \ge s_2$.  By
choosing the mesh size $h$ to be sufficiently small, we can make $s_1-s_2$
small, providing a good estimate for $s_*$.

Before describing how to construct the matrices $A_s$ and $B_s$, we need
to recall some standard results about bilinear interpolation.
On the mesh square
\begin{equation*}
R_{k,l} = \{(x,y): x_k \le x \le x_{k+1}, y_l \le y \le y_{l+1}\},
\end{equation*}
where $x_{k+1} - x_k = y_{l+1} - y_l =h$, the bilinear interpolant
$f^I(x,y)$ of a function $f(x,y)$ is given by:
\begin{multline*}
f^I(x,y) = \Big[\frac{x_{k+1} -x}{h}\Big]
 \Big[\frac{y_{l+1} -y}{h}\Big]f(x_k,y_l)
+  \Big[\frac{x- x_{k}}{h}\Big]
\Big[\frac{y_{l+1} -y}{h}\Big] f(x_{k+1},y_l)
\\
+  \Big[\frac{x_{k+1} -x}{h}\Big]
\Big[ \frac{y - y_{l}}{h}\Big] f(x_k,y_{l+1})
+  \Big[\frac{x- x_{k}}{h}\Big]
 \Big[\frac{y - y_{l}}{h} \Big] f(x_{k+1},y_{l+1}).
\end{multline*}
The error in bilinear interpolation satisfies for all $(x,y) \in R_{k,l}$
and some points $(a_k,b_l)$ and $(c_k,d_l) \in  R_{k,l}$,
\begin{multline*}
f^I(x,y) - f(x,y)
= 1/2)\Big[ (x_{k+1} -x)(x- x_{k}) (D_{xx}f)(a_k,b_l)
\\
+ (y_{l+1} -y)(y - y_{l}) (D_{yy} f)(c_k,d_l)\Big].
\end{multline*}

For $z = x+\irm y$, let $f(x,y) = w_s(\theta_{\be}(z))$. Further let $z_{k,l} =
x_k + \irm y_l$.  If $(\tilde x, \tilde y) = (\Re \theta_{\be}(z), \Im
\theta_{\be}(z)) \in R_{k,l}$, (which we will sometimes abbreviate by
$\theta_{\be}(z) \in R_{k,l}$), we get
\begin{multline*}
w_s^I(\theta_{\be}(z)) = 
\Big[\frac{x_{k+1} -\tilde x}{h}\Big]
 \Big[\frac{y_{l+1} -\tilde y}{h}\Big]w_s(z_{k,l})
+  \Big[\frac{\tilde x- x_{k}}{h}\Big]
\Big[\frac{y_{l+1} -\tilde y}{h}\Big] w_s(z_{k+1,l})
\\
+  \Big[\frac{x_{k+1} -\tilde x}{h}\Big]
\Big[ \frac{\tilde y - y_{l}}{h}\Big] w_s(z_{k,l+1})
+  \Big[\frac{\tilde x- x_{k}}{h}\Big]
 \Big[\frac{\tilde y - y_{l}}{h} \Big] w_s(z_{k+1,l+1}).
\end{multline*}
Defining
\begin{equation*}
\Psi_{\be}(z) = 1/(\bar z + \bar \be),
\end{equation*}
we have an analogous formula for $w_s^I(\Psi_{\be}(z))$, with
$(\tilde x, \tilde y) = (\Re \Psi_{\be}(z), \Im  \Psi_{\be}(z))$.

We next use inequalities \eqref{wrelation}--\eqref{Dyybound}
to obtain bounds on the interpolation error. 
By \eqref{Dxxbound} and \eqref{Dyybound}, we find
for $\theta_{\be}(z) = \tilde x + \irm \tilde y$, where $(\tilde x, \tilde
y) \in R_{k,l}$,
\begin{multline*}
-\Big[\frac{s}{8\gamma^2(s+1)} + \frac{s}{\gamma^2}\Big] 
\left([x_{k+1} - \tilde x][\tilde x -x_{k}]w_s(a_k,b_l) 
+ [y_{l+1} - \tilde y][\tilde y -y_{l}]w_s(c_k,d_l) \right) 
\\
\le  w_s^I(\theta_{\be}(z)) - w_s(\theta_{\be}(z)) 
\\
\le \frac{s(2s+1)}{\gamma^2}
\left([x_{k+1} - \tilde x][\tilde x -x_{k}]w_s(a_k,b_l)
+ [y_{l+1} - \tilde y][\tilde y -y_{l}]w_s(c_k,d_l)\right).
\end{multline*}
Applying \eqref{wrelation}, we then obtain
\begin{multline*}
- \frac{s}{\gamma^2}\Big[\frac{9+8s}{8(s+1)}\Big] 
\left([x_{k+1} - \tilde x][\tilde x -x_{k}] 
+ [y_{l+1} - \tilde y][\tilde y -y_{l}] \right) 
\exp\big(\frac{\sqrt{10} sh}{\gamma}\big) w_s^I(\theta_{\be}(z))
\\
\le  w_s^I(\theta_{\be}(z)) - w_s(\theta_{\be}(z)) 
\\
\le \frac{s(2s+1)}{\gamma^2}
\left([x_{k+1} - \tilde x][\tilde x -x_{k}]
+ [y_{l+1} - \tilde y][\tilde y -y_{l}]\right) 
\exp\big(\frac{\sqrt{10} sh}{\gamma}\big)
w_s^I(\theta_{\be}(z)).
\end{multline*}
since any point in $R_{k,l}$ is within $\sqrt{2}h$
of each of the four corners of the square $R_{k,l}$.  An analogous result
holds for $w_s(\Psi_{\be}(z))$.

Using this estimate, we have precise upper and lower bounds on the error
in the mesh square $R_{k,l}$ that only depend on the function values of
$w_s$ at the four corners of the square and the value of $\be$.  Letting
\begin{align*}
\err_{\be}^1(\theta_{\be}(z)) &=\Big([x_{k+1} - \tilde x][\tilde x -x_{k}]
+ [y_{l+1} - \tilde y][\tilde y -y_{l}]\Big)
\frac{s (2s+1)}{\gamma^2} \exp(\sqrt{10}sh/\gamma),
\\
\err_{\be}^2(\theta_{\be}(z)) &=\Big([x_{k+1} - \tilde x][\tilde x -x_{k}]
+ [y_{l+1} - \tilde y][\tilde y -y_{l}]\Big)
\frac{s}{\gamma^2}\Big[\frac{9+8s}{8+8s}\Big] 
\exp(\sqrt{10}sh/\gamma),
\end{align*}
(where again $\theta_{\be}(z) = \tilde x + \irm \tilde y$),
we have for each mesh point $z_{i,j} = x_i + \irm y_j$, with
$\theta_{\be}(z_{i,j}) \in R_{k,l}$,
\begin{equation*}
[1 - \err_{\be}^1(z_{i,j})] w_s^I(\theta_{\be}(z_{i,j})) 
\le w_s(\theta_{\be}(z_{i,j})) 
\le [1 + \err_{\be}^2(z_{i,j})]w_s^I(\theta_{\be}(z_{i,j})).
\end{equation*}
Again, the analogous result holds for $w_s(\Psi_{\be}(z))$.

To obtain the upper and lower matrices, we first note that for each
mesh point $z_{i,j}$,
\begin{multline*}
\alpha w_s(0) + \sum_{\be \in \B_R}\frac{1}{|z_{i,j}+b|^{2s}}
[1 - \err_{\be}^1(z_{i,j})] w_s^I(\theta_{\be}(z_{i,j})) 
\\
\le \sum_{\be \in \B_R} \frac{1}{|z_{i,j}+b|^{2s}}w_s(\theta_b(z_{i,j}) )
+ \alpha w_s(0) 
\\
\le \sum_{\be \in \B_R} \frac{1}{|z_{i,j}+b|^{2s}} 
[1 + \err_{\be}^2(z_{i,j})]w_s^I(\theta_{\be}(z_{i,j}))
+ \alpha w_s(0).
\end{multline*}
Motivated by the above inequality, we now define matrices
$A_s$ and $B_s$ which have nonnegative entries and satisfy the property that
$r(A_s) \le r(L_s) \le r(B_s)$. For clarity, we do this in several steps.
For $f$ a continuous, piecewise bilinear function defined on the mesh
domain $D_h$, we first define operators $\Ab_s$ and $\Bb_s$ (which also
depend on $\alpha$) by:
\begin{align}
\label{Arow}
(\Ab_s f)(z_{i,j}) &= \sum_{b \in \B_R}\frac{1}{|z_{i,j}+b|^{2s}}
[1 - \err_{\be}^1(z_{i,j})] f(\theta_{\be}(z_{i,j})) + \alpha f(0),
\\
\label{Brow}
(\Bb_s f)(z_{i,j}) &= \sum_{b \in \B_R}
\frac{1}{|z_{i,j}+b|^{2s}} 
[1 + \err_{\be}^2(z_{i,j})] f(\theta_{\be}(z_{i,j})) + \alpha f(0),
\end{align}
where $z_{i,j}$ is a mesh point in $D_h$. 
In the above, if $(\tilde x, \tilde y) = (\Re \theta_{\be}(z), \Im
\theta_{\be}(z)) \in R_{k,l}$, then, using bilinearity,
\begin{multline}
\label{bilinearity}
f(\theta_{\be}(z)) = 
\Big[\frac{x_{k+1} -\tilde x}{h}\Big]
 \Big[\frac{y_{l+1} -\tilde y}{h}\Big]f(z_{k,l})
+  \Big[\frac{\tilde x- x_{k}}{h}\Big]
\Big[\frac{y_{l+1} -\tilde y}{h}\Big]f(z_{k+1,l})
\\
+  \Big[\frac{x_{k+1} -\tilde x}{h}\Big]
\Big[ \frac{\tilde y - y_{l}}{h}\Big]f(z_{k,l+1})
+  \Big[\frac{\tilde x- x_{k}}{h}\Big]
 \Big[\frac{\tilde y - y_{l}}{h} \Big]f(z_{k+1,l+1}).
\end{multline}

Let $Q = \{z_{i,j} : z_{i,j} \text{ is a mesh point of } D_h\}$ and consider
the finite dimensional vector space $C_{\R}(Q)$. We can consider $f$ above as
an element of $C_{\R}(Q)$, where $f(\theta_{\be}(z))$ is defined by 
\eqref{bilinearity}. If we use \eqref{bilinearity} in \eqref{Arow}
and \eqref{Brow}, $\Ab_s$ and $\Bb_s$ define linear maps of 
$C_{\R}(Q)$ to $C_{\R}(Q)$. Note that
since $\err_{\be}^i = O(h^2)$ for $i=1,2$, $\Ab_s(S+) \subset S+$ and 
$\Bb_s(S+) \subset S+$ for $h$ sufficiently small, where $S+$ denotes the set
of nonnegative functions in $C_{\R}(Q)$. If, for fixed $\alpha \ge 0$, we take
$f = w_s$, the strictly positive eigenfunction of $L_{s,R, \alpha}$, our
construction insures that for all mesh points $z_{i,j} \in D_h$,
\begin{equation*}
(\Ab_s w_s)(z_{i,j}) \le (L_{s,R, \alpha} w_s)(z_{i,j})
= r(L_{s,R, \alpha}) w_s(z_{i,j}) \le (\Bb_s w_s)(z_{i,j}).
\end{equation*}
Lemma~\ref{lem:nonneg} now implies that
\begin{equation}
\label{nnimplies}
r(\Ab_s) \le r(L_{s,R, \alpha}) \le r(\Bb_s).
\end{equation}
If $\B$ is finite, so $\alpha =0$ and $L_{s,R} = L_s$, \eqref{nnimplies}
gives an estimate for $r(L_s)$ in terms of the spectral radii of finite
dimensional linear maps $\Ab_s$ and $\Bb_s$.  If $\B$ is infinite and $R > 0$
has been chosen and $\eta_{s,R}$ and $\delta_{s,R}$ have been estimated as in
Theorems~\ref{thm:thm5.12} and \ref{thm:thm5.13}, we take $\alpha = \eta_{s,R}$
in \eqref{Arow} and define $\Ab_s$ as in \eqref{Arow} and we obtain, using
Theorem~\ref{thm:thm3.3},
\begin {equation}
\label{lowbound}
r(\Ab_s) \le r(L_{s,R-}) \le r(L_s).
\end{equation}
Taking $\alpha = \delta_{s,R}$
in \eqref{Brow}, we define $\Bb_s$ as in \eqref{Brow} to obtain
\begin {equation}
\label{upbound}
r(L_s) \le r(L_{s,R+}) \le r(\Bb_s).
\end{equation}

As a practical matter, it remains to describe the linear maps $\Ab_s$ and
$\Bb_s$ as matrices.  For simplicity, we totally order the elements of $Q$ by
the dictionary ordering, i.e., $z_{i,j} < z_{p,q}$ if and only if $i < p$ or
if $i=p$ and $j < q$. Then we can identify $f \in C_{\R}(Q)$ with a column
vector $(f_1, \ldots, f_k, \ldots, f_n)^T$, where $f(z_{i,j}):= f_k$ if
$z_{i,j}$ is the $k$th element when the mesh points in $D_h$ are ordered as
above and $n$ is the total number of mesh points in $D_h$, Since
$f(\theta_b(z))$ is a linear combination of four components of $f$, the mesh
point $z_{i,j}$ will produce row $k$ of the matrix $A_s$ (and similarly for
$B_s$).  A more detailed description of this procedure can be found in
\cite{hdcomp1} for a simpler one dimensional problem. 
Since $A_s$ and $B_s$ are just
representations of the linear maps $\Ab_s$ and $\Bb_s$, we can replace
$r(\Ab_s)$ by $r(A_s)$ in \eqref{lowbound} and $r(\Bb_s)$ by $r(B_s)$ in
\eqref{upbound}.  Thus, we can restate \eqref{lowbound} and 
\eqref{upbound} in terms of the spectral radii of the matrices
$A_s$ and $B_s$, which better conforms to the description in
Section~\ref{sec:intro}:
\begin {equation*}
r(A_s) \le r(L_s) \le r(B_s).
\end{equation*}


As described in Section~\ref{sec:intro}, if $s_*$ denotes the unique value of
$s$ such that $r(L_{s_*}) = \lambda_{s_*} = 1$, then $s_*$ is the Hausdorff
dimension of the invariant set under study.  Hence, if we can find a number
$s_1$ such that $r(B_{s_1}) \le 1$, then $r(L_{s_1}) \le r(B_{s_1}) \le 1$,
and we can conclude that $s_* \le s_1$.  Analogously, if we can find a number
$s_2$ such that $r(A_{s_2}) \ge 1$, then $r(L_{s_2}) \ge r(A_{s_2}) \ge 1$,
and we can conclude that $s_* \ge s_2$.  By choosing the mesh sufficiently
fine and both $r(B_{s_1})$ and $r(A_{s_2})$ very close to one, we can make
$s_1-s_2$ small, providing a good estimate for $s_*$. As noted in
Section~\ref{sec:intro}, since the map $s \mapsto r(L_{s,R,\alpha})$ is
log convex, we can find the desired values of $s_1$ and $s_2$ by using a
slight modification of the secant method applied to finding zeros of the
functions $\log[r(A_{s_2})]$ and $\log[r(B_{s_2})]$.  We also note that since
the matrices $A_s$ and $B_s$ will have a single dominant eigenvalue, (see
Section~\ref{sec:compute-sr} of this paper and Section 7 of \cite{hdcomp1}),
the spectral radius is easily computed by a variant of the power method (in
fact, our computer codes simply call the {\it Matlab} routine {\tt eigs}).
Indeed, the same program also gives high order approximations to the strictly
positive eigenvectors associated to $r(A_s)$ and $r(B_s)$.

By our remarks above, it only remains to use our estimates for $\eta_{s,R}$
and $\delta_{s,R}$ in \eqref{specrel1} and \eqref{specrel2} when $\B$ is
infinite, since then we will have the matrices $A_s$ and $B_s$.

In Table~\ref{tb:t4}, we present the computation of upper and
  lower bounds for the Hausdorff dimension of the invariant sets associated to
$\B = I_1, I_2$, and $I_3$.  In the table, we study the effects of decreasing
the mesh size $h$ and increasing the value of $R$, which corresponds
to only including terms in the sum for which $|b| \le R$.  Each row in the
table gives upper and lower bounds, and for $R$ fixed, one can see that
the lower bounds are increasing and the upper bounds decreasing as
$h$ is decreased.  Similarly, taking a larger value of $R$ improves the
bounds for the same mesh size.  Except for possible round off error in
these calculations, which we do not expect to affect the results for
the number of decimal places shown, our theorems prove that these are
in fact upper and lower bounds for the actual Hausdorff dimension.

\begin{table}[!ht]
\caption{Computation of Hausdorff dimension $s$  
for several values of $h$ and $R$ (rounded to 5 decimal places).}
\label{tb:t4}
\begin{center}
\begin{tabular}{|c|c|c|c|c|}
\hline
Set & $h$ & $R$    &   lower $s$  & upper $s$ \\
\hline \hline
$I_1$ &  $.02$ & $100$ &  1.85516   & 1.85608\\
$I_1$ &  $.01$ & $100$ &  1.85563   & 1.85594 \\
$I_1$ &  $.005$ & $100$ & 1.85574    & 1.85590 \\
\hline
$I_1$ &  $.02$ & $200$ &  1.85521   & 1.85604 \\
$I_1$ &  $.01$ & $200$ &  1.85568   & 1.85589 \\
\hline
$I_1$ &  $.02$ & $300$ & 1.85522   & 1.85603 \\
\hline \hline
$I_2$ &  $.02$ & $100$ &  1.48883    &   1.49010 \\
$I_2$ &  $.01$ & $100$ &  1.48904   & 1.49003 \\
$I_2$ &  $.005$ & $100$ & 1.48909  &  1.49002 \\
\hline
$I_2$ &  $.02$ & $200$ & 1.48925    & 1.48985  \\
$I_2$ &  $.01$ & $200$ &  1.48946   &  1.48978 \\
\hline
$I_2$ &  $.02$ & $300$ & 1.48933    &  1.48981 \\
\hline \hline
$I_3$ &  $.02$  &  & 1.53706 &  1.53790  \\
$I_3$ &  $.01$  &  & 1.53754 &  1.53774  \\
$I_3$ &  $.005$ &  & 1.53765 &  1.53770  \\
\hline
\end{tabular}
\end{center}
\end{table}


\begin{remark}
\label{rem:interval}
It is important to note that, given $s_1$ and $s_2$, $B_{s_1}$ and $A_{s_2}$
are, modulo roundoff errors in computation, known exactly.  Furthermore, our
computer program furnishes (purported) strictly positive eigenvectors
$w_{s_1}$ for $B_{s_1}$  and $u_{s_2}$ for $A_{s_2}$, with respective
eigenvalues $r(B_{s_1}) < 1$ and $r(A_{s_2}) > 1$.  However, we do not need
to know whether $w_{s_1}$ and $u_{s_2}$  are actually eigenvectors. It
suffices to verify that
\begin{equation}
\label{verify}
B_{s_1} w_{s_1} \le w_{s_1} \qquad \text{and} \qquad 
A_{s_2} u_{s_2} \ge u_{s_2},
\end{equation}
since then Lemma~\ref{lem:nonneg} implies that $r(B_{s_1}) \le 1$ and
$r(A_{s_2}) \ge 1$, and we obtain that $s_2 \le s_* \le s_1$.  The vectors
$u_{s_2}$ and $w_{s_1}$ are given to us exactly by the program.  We have
verified \eqref{verify} to high accuracy, but we have not used interval
arithmetic. If we had used interval arithmetic to calculate $B_{s_1}$,
$A_{s_2}$, and to verify \eqref{verify}, the estimates in Table~\ref{tb:t4}
would be completely rigorous.  It is in that sense that we list the following
result as a theorem.
\end{remark}
\begin{thm}
\label{thm:thm-rigorous}
The Hausdorff dimensions of the invariant sets associated to $\B = I_1$, $I_2$,
and $I_3$ satisfy the bounds
\begin{gather*}
I_1: \quad 1.85574 \le s \le 1.85589, \qquad
I_2: \quad 1.48946 \le s \le 1.48978,
\\
I_3: \quad 1.53765 \le s \le 1.53770.
\end{gather*}
\end{thm}

\subsection{Higher order approximation}
\label{higher-order}

Although the theory developed in this paper does not apply to higher order
piecewise polynomial approximation, since one cannot guarantee that the
approximate matrices have nonnegative entries, we also report in
Table~\ref{tb:t5} and Table~\ref{tb:t6} the results of higher order piecewise
polynomial approximation to demonstrate the promise of this
approach.   In this case, we only provide the results for the approximate
matrix, which does not contain any corrections for the interpolation error.

Since we did not have an exact solution for the problem corresponding to the
set $I_3$, we cannot compare the actual errors. However, assuming the last
entry in Table~\ref{tb:t5} gives the most accurate approximation, we see that
the third entry using piecewise cubics is accurate to 10 decimal places, which
is a significant improvement over the last entry for linear approximation,
which only produces 5 correct digits after the decimal point.  This is
consistent with the theory of approximation of smooth functions by piecewise
polynomials, which shows that the convergence rate grows as the degree of the
polynomials is increased.  In the computations shown using higher order
piecewise polynomials, to get a fair comparison, we have adjusted the mesh
sizes so that the results for different degree piecewise polynomials will have
approximately the same number of degrees of freedom (DOF).

\begin{table}[!ht]
\caption{Computation of Hausdorff dimension $s$ of the set $I_3$
using higher order piecewise polynomials.}
\label{tb:t5}
\begin{center}
\begin{tabular}{|c|c|c|c|c|c|c|}
\hline
degree & h  & \# DOF &  $s$  \\
\hline \hline 
1 & .02 & 1098 & 1.537729111247678  \\
1 & .01 & 4165 & 1.537694920731214  \\
1 & .005 & 16201 &  1.537686565250360 \\
\hline
2 & 0.041667 & 1041 & 1.537683708302400    \\
2 & 0.020833 & 3913 &  1.537683729607203 \\
2 &  0.010417 & 15089 &  1.537683732415111\\
\hline
3 & 0.0625 & 1081 & 1.537683753797206  \\
3 & 0.03125 & 3997 &   1.537683734167568 \\
3 & 0.015625 & 15283 &  1.537683732983929 \\
3 & 0.0078125 & 59545 &   1.537683732912027 \\
\hline 
\end{tabular}
\end{center}
\end{table}

In a future paper we hope to prove that rigorous upper and lower bounds for
the Hausdorff dimension can also be obtained when higher order piecewise
polynomial approximations are used.

\subsection{A special example with a known solution}
\label{sec:special}
To further test the algorithm, especially using higher order
piecewise polynomials, we constructed a special example where
the exact solution is known.  More specifically, we considered the operator
\begin{equation*}
(L_s(f))(z) = \sum_{b \in \B} g_b^s(z) f(\theta_b(z)),
\end{equation*}
where $\B = \{1 \pm \irm, 2 \pm \irm, 3 \pm \irm\}$ and
\begin{equation*}
g_b(z) = \frac{1}{6}\Big|\frac{z+b+1}{z+b}\Big|^2 \Big|\frac{1}{z+1}\Big|^2.
\end{equation*}
This example is constructed so that $f(z) = |1/(z+1)|^2$ is an eigenfunction
of $L_1$ with eigenvalue $\lambda = 1$ for $s=1$.  In Table~\ref{tb:t6}, we
present the results of approximations using different values of $h$ and
different degree piecewise polynomials.

\begin{table}[!ht]
  \caption{Approximation, using higher order piecewise polynomials,
of the number $s=1$ for which $r(L_s) =1$ for the special example.}
\label{tb:t6}
\begin{center}
\begin{tabular}{|c|c|c|c|c|c|c|}
\hline
degree & h  & \# DOF &  $s$  \\
\hline \hline 
1 & .02 & 1098 & 1.000034749616189  \\
1 & .01 & 4165 & 1.000010815423902  \\
1 & .005 & 16201 &   1.000002596942892               \\
\hline
2 & .02 & 4239 &  1.000000016815596 \\
2 & .01 & 16357 & 0.999999997912829      \\
\hline
3 & .02 & 9424 &  1.000000000610834  \\
\hline
4 & .04167 & 4017 & 0.999999999999715 \\
4 & .02 & 16653 & 0.999999999999925   \\
\hline 
\end{tabular}
\end{center}
\end{table}

\section{Existence of $C^m$ positive eigenfunctions}
\label{sec:exist}
In this section we shall describe some results concerning existence of
$C^m$ positive eigenfunctions for a class of positive (in the sense of
order-preserving) linear operators.  We shall later indicate how one
can often obtain explicit bounds on partial derivatives of the
positive eigenfunctions.  As noted above, such estimates play a crucial
role in our numerical method and therefore in obtaining rigorous
estimates of Hausdorff dimension for invariant sets associated with
iterated function systems.

The starting point of our analysis is Theorem 5.5 in \cite{E}, which we
now describe for a simple case. If $H$ is a bounded open subset of
$\R^n$ and $m$ is a positive integer, $C^m_{\C}(\bar H)$ will denote the
set of complex-valued $C^m$ maps $f:H \to \C$ such that all partial
derivatives $D^{\alpha} f$ with $|\alpha| \le m$ extend continuously
to $\bar H$. (Here $\alpha = (\alpha_1, \ldots, \alpha_n)$ is a
multi-index with $\alpha_j \ge 0$ for all $j$, $D_j
= \partial/\partial x_j$ for $1 \le j \le n$ and $D^{\alpha} =
D_1^{\alpha_1} \cdots D_n^{\alpha_n}$), $C^m_{\C}(\bar H)$ is a complex Banach
space with $\|f\| = \sup\{|D^{\alpha} f(x)|: x \in H, |\alpha| \le
m\}$. Analogously, $C^m_{\R}(\bar H)$ denotes the corresponding real Banach
space of real-valued $C^m$ maps $f: H \to \R$.

We say that $H$ is {\it mildly regular} if there exist $\eta >0$ and
$M \ge 1$ such that whenever $x,y \in H$ and $\|x-y\| < \eta$, there exists
a Lipschitz map $\psi:[0,1] \to H$ with $\psi(0) = x$, $\psi(1) = y$
and
\begin{equation}
\label{1.1}
\int_0^1 \|\psi^{\prime}(t)\| \, dt \le M \|x-y\|.
\end{equation}
(Here $\|\cdot\|$ denotes any fixed norm on $\R^n$. If the norm is changed,
\eqref{1.1} remains valid, but with a different constant $M$.)

Let $\B$ denote a finite index set with $|\B| = p$.  For $\be \in \B$, we
assume
\begin{align*}
&\text{(H4.1)} \ \,
g_{\be} \in C^m_{\R}(\bar H) \text{ for all } \be \in \B \text{ and }
g_{\be}(x) > 0 \text{ for all } x \in \bar H \text{ and all } \be \in \B.
\\
&\text{(H4.2)} \ \,
\theta_{\be}:H \to H \text{ is a } C^m \text{ map for all } \be \in \B,
\text{ i.e., if } \theta_{\be}(x) = (\theta_{\be_1}(x),
\ldots \theta_{\be_n}(x)),
\\
&\ \, \qquad \qquad \text{ then } \theta_{\be_k} \in C^m_{\R}(\bar H)
\text{ for all } \be \in \B \text{ and for } 1 \le k \le n.
\end{align*}
In (H4.1) and (H4.2), we always assume that $m \ge 1$.

We define a bounded, complex linear map $\Lambda: C^m_{\C}(\bar H) \to 
C^m_{\C}(\bar H)$ by
\begin{equation}
\label{1.2}
(\Lambda(f))(x) = \sum_{\be \in B} g_{\be}(x) f(\theta_{\be}(x)).
\end{equation}
Equation \eqref{1.2} also defines a bounded real linear map of 
$C^m_{\R}(\bar H)$ to itself which we shall also denote by $\Lambda$.

For integers $\mu \ge 1$, we define $\B_{\mu} := \{\w = (j_1, \ldots j_{\mu})
: j_k \in \B \text{ for } 1 \le k \ \le \mu\}$. For
$\w = (j_1, \ldots j_{\mu}) \in \B_{\mu}$, we define $\w_{\mu} = \w$,
$\w_{\mu -1} = (j_1, \ldots j_{\mu-1})$,
$\w_{\mu -2} = (j_1, \ldots j_{\mu-2})$, $\cdots$, $\w_1 = j_1$.  We define
\begin{equation*}
\theta_{\w_{\mu-k}}(x) = (\theta_{j_{\mu-k}} \circ \theta_{j_{\mu-k-1}} \circ
\cdots \circ \theta_{j_{1}})(x),
\end{equation*}
so
\begin{equation*}
\theta_{\w}(x):= \theta_{\w_{\mu}}(x) = 
(\theta_{j_{\mu}} \circ \theta_{j_{\mu-1}} \circ
\cdots \circ \theta_{j_{1}})(x).
\end{equation*}

For $\w \in \B_{\mu}$, we define $g_{\w}(x)$ inductively by 
$g_{\w}(x) = g_{j_1}(x)$
if $\w = (j_1) \in \B:=\B_1$,
$g_{\w}(x) = g_{j_2}(\theta_{j_1}(x)) g_{j_1}(x)$ if $\w = (j_1,j_2) \in \B_2$
and, for $\w = (j_1,j_2, \ldots j_{\mu}) \in \B_{\mu}$,
\begin{equation*}
g_{\w}(x) = g_{j_\mu}(\theta_{\w_{\mu-1}}(x)) g_{\w_{\mu -1}}(x).
\end{equation*}

If is not hard to show (see \cite{A}, \cite{Bourgain-Kontorovich}, \cite{E})
that
\begin{equation}
\label{1.6}
(\Lambda^{\mu}(f))(x) = \sum_{\w \in \B_{\mu}} g_{\w}(x) f(\theta_\w(x)).
\end{equation}

If $\Lambda$ and $m$ are as above, we shall let $\sigma(\Lambda) \subset \C$
denote the spectrum of $\Lambda$.  If all the functions $g_j$ and $\theta_j$
are $C^N$, then we can consider $\Lambda$ as a bounded linear operator
$\Lambda_m: C^m_{\C}(\bar H) \to C^m_{\C}(\bar H)$ for $1 \le m \le N$, but
one should note that in general $\sigma(\Lambda_m)$ will depend on $m$.

To obtain a useful theory for $\Lambda$, we need a further crucial assumption.
For a given norm $\|\cdot \|$ on $\R^n$, we assume

(H4.3)  There exists a positive integer $\mu$ and a constant $\kappa <1$ such
that for all $\w \in \B_{\mu}$ and all $x,y \in H$,
\begin{equation*}
\|\theta_\w(x) - \theta_\w(y)\| \le \kappa \|x-y\|.
\end{equation*} 

If we define $c = \kappa^{1/\mu} <1$, it follows from (H4.3) that there
exists a constant $M$ such that for all $\w \in B_{\nu}$ and all $\nu \ge 1$,
\begin{equation}
\label{1.8}
\|\theta_\w(x) - \theta_\w(y)\| \le M c^{\nu} \|x-y\| \quad \forall x,y \in H.
\end{equation}
If the norm $\|\cdot \|$ in \eqref{1.8} is replaced by a different norm
$|\cdot |$, \eqref{1.8} remains valid, although with a different constant $M$.
This in turn implies that (H4.3) will also be valid with the same constant
$\kappa$, with $|\cdot|$ replacing $\|\cdot\|$ and with a possibly different
integer $\mu$.

The following theorem is a special case of Theorem 5.5 in \cite{E}.

\begin{thm}
\label{thm:1.1} 
Let $H$ be a bounded open subset of $\R^n$ and assume that $H$ is mildly
regular. Let $X = C^m_{\C}(\bar H)$ and assume that (H4.1), (H4.2), and (H4.3)
are satisfied (where $m \ge 1$ in (H4.1) and (H4.2)) and that $\Lambda:X \to
X$ is given by \eqref{1.2}.  If $Y= C_{\C}(\bar H)$, the Banach space of
complex-valued continuous functions $f: \bar H \to \C$ and $L:Y \to Y$ is
defined by \eqref{1.2}, then $r(L) = r(\Lambda) >0$, where $r(L)$ denotes the
spectral radius of $L$ and $r(\Lambda)$ denotes the spectral radius of
$\Lambda$.  If $\rho(\Lambda)$ denotes the essential spectral radius of
$\Lambda$ (see \cite{B},\cite{A},\cite{N-P-L}, and \cite{L}), then
$\rho(\Lambda) \le c^m r(\Lambda)$ where $c= \kappa^{1/\mu}$ is as in
\eqref{1.8}.  There exists $v \in X$ such that $v(x) >0$ for all $x \in \bar
H$ and
\begin{equation*}
\Lambda(v) = r v, \qquad r = r(\Lambda).
\end{equation*}
There exists $r_1 < r$ such that if $\xi \in \sigma(\Lambda)
\setminus\{r\}$,
then $|\xi| \le r_1$; and $r = r(\Lambda)$ is an isolated point of
$\sigma(\Lambda)$ and an eigenvalue of algebraic multiplicity 1. If $u \in X$
and $u(x) >0 \, \forall x \in \bar H$, there exists a real number $s_u >0$ such
that
\begin{equation}
\label{1.10}
\lim_{k \rightarrow \infty}\left(\frac{1}{r} \Lambda\right)^k (u) = s_u v,
\end{equation}
where the convergence in \eqref{1.10} is in the $C^m$ topology on $X$.
\end{thm}

\begin{remark}
\label{rem:1.2} If $\alpha$ is a multi-index with 
$|\alpha|\le m$, where $m \ge 1$
is as in (H4.1) and (H4.2), it follows from \eqref{1.10} that
\begin{equation}
\label{1.11}
\lim_{k \rightarrow \infty} \left(\frac{1}{r}\right)^k D^{\alpha} \Lambda^k(u) 
=s_u D^{\alpha} v,
\end{equation}
and
\begin{equation}
\label{1.12}
\lim_{k \rightarrow \infty} \left(\frac{1}{r}\right)^k \Lambda^k(u) 
=s_u v,
\end{equation}
where the convergence in \eqref{1.11} and \eqref{1.12} is in the topology of
$C_{\C}(\bar H)$, the Banach space of continuous functions $f: \bar H \to \C$.
\end{remark}
It follows from \eqref{1.11} and \eqref{1.12} that for any
multi-index $\alpha$ with $|\alpha| \le m$,
\begin{equation}
\label{1.13}
\lim_{k \rightarrow \infty} \frac{(D^{\alpha} \Lambda^k(u))(x)}
{\Lambda^k(u)(x)} = \frac{(D^{\alpha} (v))(x)}
{v(x)},
\end{equation}
where the convergence in \eqref{1.13} is uniform in $x \in \bar H$.
If we choose $u(x) =1$ for all $x \in \bar H$, it follows from \eqref{1.6}
that for all multi-indices $\alpha$ with $|\alpha| \le m$, we have
\begin{equation}
\label{1.14}
\lim_{k \rightarrow \infty}  \frac{D^{\alpha} (\sum_{\w \in B_k} g_\w(x))}
{\sum_{\w \in B_k} g_\w (x)} = \frac{D^{\alpha} v(x)}{v(x)},
\end{equation}
where the convergence in \eqref{1.14} is uniform in $x \in \bar H$. We shall
use \eqref{1.14} in our further work to obtain explicit bounds on 
$\sup\left\{|D^{\alpha} v(x)|/v(x): x \in \bar H\right\}$.

Direct analogues of Theorem 5.5 in \cite{E} exist when $\B$ is countable but
not finite, but such analogues were not stated or proved in \cite{E}.  We
shall make do here with less precise theorems which we shall prove by an {\it
  ad hoc} argument in the next section. We refer to Lemma~5.3 in Section 5 of
\cite{N-P-L}, Theorem~5.3 on p. 86 of \cite{A} and Section 5 of \cite{A} for
more information about existence of positive eigenfunctions when $\B$ is
infinite.

\section{The Case of M\"obius Transformations}
\label{sec:mobius}
By working with partial derivatives and using methods like those in Section 5
of \cite{hdcomp1}, it is possible to obtain explicit estimates on partial
derivatives of $v_s(x)$ in the generality of Theorem~\ref{thm:1.1}.  However,
for reasons of length and in view of the immediate applications in this paper,
we shall not treat the general case here and shall now specialize to the case
that the mappings $\theta_{\be}(\cdot)$ are given by M\"obius
transformations which map a given bounded open subset $H$ of $\C:=\R^2$ into
$H$. Specifically, throughout this section we shall usually assume:

\noindent (H5.1): $\gamma \ge 1$ is a given real number and $\B$ is a finite
collection of complex numbers $\be$ such that $\Re(\be) \ge \gamma$ for
all $\be \in \B$. For each $\be \in \B$, $\theta_{\be}(z):= 1/(z +
\be)$ for $z \in \C \setminus \{- \be\}$.

The assumption in (H5.1) that $\gamma \ge 1$ is only a convenience; and the
results of this section can be proved under the weaker assumption that $\gamma
>0$.

For $\gamma >0$ we define $G_{\gamma} \in \C$ by
\begin{equation} 
\label{4.1}
G_{\gamma} = \{ z \in \C: |z - 1/(2 \gamma)|
< 1/(2 \gamma) \}.
\end{equation}
It is easy to check that if $w \in \C$ and $\Re(w) > \gamma$, then
$(1/w) \in G_{\gamma}$.  It follows that if $\Re(z) >0$, $\be \in
\C$ and $\Re(\be) \ge \gamma >0$, then $\theta_{\be}(z) \in \bar
G_{\gamma}$. Let $H$ be a bounded, open, mildly regular subset of $\C
= \R^2$ such that $H \supset G_{\gamma}$ and $H \subset \{z : \Re(z)
>0\}$, and let $\B$ denote a finite set of complex numbers such that
$\Re(\be) \ge \gamma >0$ for all $\be \in \B$. We define a bounded
linear map $\Lambda_s:C^m_{\C}(\bar H) \to C^m_{\C}(\bar H)$, where $m$ is a
positive integer and $s \ge 0$, by
\begin{equation} 
\label{4.2}
(\Lambda_s(f))(z) = \sum_{\be \in\B} \Big|\frac{d}{dz}
\theta_{\be}(z)\Big|^s f(\theta_{\be}(z)):= \sum_{\be \in\B} 
\frac{1}{|z+ \be|^{2s}} f(\theta_{\be}(z)).
\end{equation}
As in Section~\ref{sec:intro}, $L_s: C_{\C}(\bar H) \to C_{\C}(\bar H)$ is
defined by \eqref{4.2}.  We use different letters to emphasize that
$\sigma(\Lambda_s) \neq \sigma(L_s)$, although $r(\Lambda_s) = r(L_s)$.

If all elements of $\B$ are real, we can restrict attention to the real line
and, as we shall see, the analysis is much simpler.  In this case we abuse
notation and take $G_{\gamma} = (0, 1/\gamma) \subset \R^2$ and $H= (0,a)$, $a
\ge 1/\gamma$.  For $f \in C^m_{\C}(\bar H)$ and $x \in \bar H$, \eqref{4.2}
takes the form
\begin{equation*}
(\Lambda_s(f))(x) = \sum_{\be \in\B} \frac{1}{(x+ \be)^{2s}}
f(\theta_{\be}(x)).
\end{equation*}

If, for $1 \le j \le n$, $M_j = \bigl( \begin{smallmatrix} a_j & b_j \\ c_j &
  d_j \end{smallmatrix} \bigr)$ is a $2 \times 2$ matrix with complex entries
and $\det(M_j) = a_j d_j - b_j c_j \neq 0$, define a M\"obius transformation
$\psi_j(z) = (a_j z + b_j)/(c_j z + d_j)$.  It is well-known that
\begin{equation}
\label{4.4}
(\psi_1 \circ \psi_2 \circ \cdots \circ \psi_n)(z) = (A_n z + B_n)/(C_n z +
D_n),
\end{equation}
where
\begin{equation}
\label{4.5}
\begin{pmatrix} A_n & B_n \\ C_n &   D_n \end{pmatrix}
= M_1 M_2 \cdots  M_n.
\end{equation}

If $\B$ is a finite set of complex numbers $\be$ such that
$\Re(\be) \ge \gamma >0$ for all $\be \in \B$, we define $\B_{\nu}$
as before by
\begin{equation*}
\B_{\nu} = \{ \w = (\be_1, \be_2, \ldots, \be_{\nu}) : \be_j \in
\B \text{ for } 1 \le j \le \nu\}
\end{equation*}
and $\theta_{\w} = \theta_{\be_n} \circ \theta_{\be_{n-1}} \cdots \circ
\theta_{\be_1}$. Given $\w = (\be_1, \be_2, \ldots, \be_{\nu}) 
\in \B_{\nu}$, we define
\begin{equation}
\label{4.6} 
\tilde \w = (\be_{\nu}, \be_{\nu-1}, \ldots, \be_{1})
\end{equation}
so
\begin{equation}
\label{4.7}
\theta_{\tilde \w} = \theta_{\be_1} \circ \theta_{\be_{2}} \cdots \circ
\theta_{\be_n}.
\end{equation}
For $\Lambda_s$ as in \eqref{4.2} $\nu \ge 1$, and $f \in C^m_{\C}(\bar H)$,
recall that
\begin{equation*}
(\Lambda_s^{\nu}(f))(z) = \sum_{\w \in \B_{\nu}}
\Big|\frac{d \theta_{\w}(z)}{dz}\Big|^s f(\theta_{\w}(z))
= \sum_{\w \in \B_{\nu}} \Big|\frac{d \theta_{\tilde \w}(z)}{dz}\Big|^s 
f(\theta_{\tilde \w}(z)).
\end{equation*}

The following lemma allows us to apply Theorem~\ref{thm:1.1} to $\Lambda_s$ in
\eqref{4.2}.

\begin{lem}
\label{lem:4.1}
Let $\be_1$ and $\be_2$ be complex numbers with $\Re(\be_j) \ge \gamma
\ge 1$ for $j =1,2$.  If $\psi_j(z) = 1/(z + \be_j)$ for
$\Re(z) \ge 0$ and $\theta = \psi_1 \circ \psi_2$, then for all $z, w$
with $\Re(z) \ge 0$ and $\Re(w) \ge 0$,
\begin{equation*}
|\theta(z) - \theta(w)| \le (\gamma^2 +1)^{-2} |z-w|.
\end{equation*}
\end{lem}
\begin{proof}
It suffices to prove that $|(d \theta/dz)(z)| \le (\gamma^2 +1)^{-2}$
for all $z \in \C$ with $\Re(z) \ge 0$.  Using \eqref{4.4} and \eqref{4.5}
we see that
\begin{equation*}
|(d \theta/dz)(z)| = |\be_1|^{-2} |z + (1/\be_1) + \be_2|^{-2},
\end{equation*}
so it suffices to prove that
$|\be_1|^2 \, |z + (1/\be_1) + \be_2|^{2} \ge (\gamma^2 +1)^{2}$
for $\Re(z) \ge 0$. If we write $\be_1 = u+ iv$ with $u \ge \gamma$,
\begin{equation*}
\Re(z + (1/\be_1) + \be_2) \ge u/(u^2 + v^2)  + \gamma,
\end{equation*}
so
\begin{equation*}
|z + (1/\be_1) + \be_2|^{2} \ge [u/(u^2 + v^2)  + \gamma]^2
\end{equation*}
and
\begin{multline*}
|\be_1|^2 \, |z + (1/\be_1) + \be_2|^{2} \ge (u^2 + v^2)
\Big[\frac{u^2}{(u^2 + v^2)^2} + \frac{2 u \gamma}{(u^2 + v^2)} + \gamma^2
\Big]
\\
= \frac{u^2}{(u^2 + v^2)} + 2 u \gamma + \gamma^2(u^2 + v^2) = g(u,v).
\end{multline*}
Because $u \ge \gamma$, $g(u,0) = 1 + 2 \gamma^2 + \gamma^4
= (\gamma^2 +1)^2$.  Using the fact that $u \ge \gamma \ge 1$, we also see
that for $v \ge 0$
\begin{equation*}
\frac{ \partial g(u,v)}{\partial v} = \frac{-u^2(2v)}{(u^2 + v^2)^2}
+ 2 \gamma^2 v \ge 0,
\end{equation*}
which implies that $g(u,v) \ge g(u,0) = (\gamma^2 +1)^{2}$ for $u \ge \gamma$
and $v \ge 0$.  Since $g(u,-v) = g(u,v)$,
$g(u,v) \ge (\gamma^2 +1)^{2}$ for $v \le 0$ and $u \ge \gamma$.
\end{proof}

With the aid of Lemma~\ref{lem:4.1}, the following theorem is an immediate
corollary of Theorem~\ref{thm:1.1}.

\begin{thm}
\label{thm:4.2}
Assume (H5.1) and let $H$ be a bounded, open mildly regular subset of
$\{z \in \C : \Re(z) >0\}$ such that $H \supset G_{\gamma}$, where
$G_{\gamma}$ is defined by \eqref{4.1}.  For a
given positive integer $m$ and for $s > 0$, let
$X = C^m_{\C}(\bar H)$ and $Y = C_{\C}(\bar H)$ and let $
\Lambda_s: X \to X$ and $L_s:Y \to Y$ be given by \eqref{4.2}.
If $r(\Lambda_s)$ (respectively, $r(L_s)$) denotes the spectral radius
of $\Lambda_s$ (respectively, $L_s$), we have $r(\Lambda_s) > 0$ and
$r(\Lambda_s) = r(L_s)$. If $\rho(\Lambda_s)$ denotes the essential
spectral radius of $\Lambda_s$,
\begin{equation*}
\rho(\Lambda_s) \le (\gamma^2 +1)^{-m} r(\Lambda_s).
\end{equation*}
For each $s >0$, there exists $v_s \in X$ such that $v_s(z) >0$ for all
$z \in \bar H$ and $\Lambda_s(v_s) = r(\Lambda_s)v_s$.  All the statements
of Theorem~\ref{thm:1.1} are true in this context whenever $\Lambda$ and $L$
in Theorem~\ref{thm:1.1} are replaced by $\Lambda_s$ and $L_s$ respectively.
\end{thm}

In the notation of Theorem~\ref{thm:4.2}, it follows from \eqref{1.14} that
for any multi-index $\alpha=(\alpha_1,\alpha_2)$ with $\alpha_1 + \alpha_2 \le
m$ and for $z = x+iy = (x,y)$
\begin{equation}
\label{4.11}
\lim_{\nu \rightarrow \infty} \frac{D^{\alpha}\left(\sum_{\w \in \B_{\nu}}
\Big|\frac{d}{dz} \theta_{\w}(z)\Big|^s\right)}
{\sum_{\w \in \B_{\nu}} \Big|\frac{d}{dz} \theta_{\w}(z)\Big|^s}
= \frac{D^{\alpha} v_s(x,y)}{v_s(x,y)},
\end{equation}
where the convergence is uniform in $(x,y):=z \in \bar H$
and $D^{\alpha} = (\partial/\partial x)^{\alpha_1}
(\partial/\partial y)^{\alpha_2}$.

\begin{lem}
\label{lem:4.3}
Let $\be_j$, $j \ge 1$ be a sequence of complex numbers with
$\Re(\be_j) \ge \gamma > 0$ for all $j$.  For complex numbers $z$,
define $\theta_{\be_j}(z) = (z + \be_j)^{-1}$ and define matrices
$M_j =
\bigl(\begin{smallmatrix} 0 & 1 \\
  1 & \be_j \end{smallmatrix}\bigr)$.  Then for $n \ge 1$,
\begin{equation}
\label{4.12}
M_1 M_2 \cdots M_n = \begin{pmatrix} A_{n-1} & A_n \\ B_{n-1} & B_n
\end{pmatrix},
\end{equation}
where $A_0 =0 $, $A_1 =1$, $B_0 =1$, $B_1 = \be_1$ and for $n \ge 1$,
\begin{equation}
\label{4.13}
A_{n+1} = A_{n-1} + \be_{n+1} A_n \text{ and }
B_{n+1} = B_{n-1} + \be_{n+1} B_n.
\end{equation}
Also,
\begin{equation*}
(\theta_{\be_1} \circ \theta_{\be_{2}} \cdots \circ \theta_{\be_n})(z)
= (A_{n-1} z + A_n)/(B_{n-1} z + B_n),
\end{equation*}
and we have
\begin{equation}
\label{4.14}
\Re(B_n/B_{n-1}) \ge \gamma
\end{equation}
and
\begin{equation}
\label{4.15}
\Big|\frac{d}{dz}\Big[\frac{A_{n-1} z + A_n}{B_{n-1} z + B_n}\Big]\Big|^s
= |B_{n-1}|^{-2s} |z + B_n/B_{n-1}|^{-2s}.
\end{equation}
\end{lem}
\begin{proof}
Equation \eqref{4.12} follows by induction on $n$. It is obviously true
for $n=1$. If we assume that \eqref{4.12} is satisfied for some $n \ge 1$,
then
\begin{equation*}
M_1 M_2 \cdots M_n M_{n+1}= \begin{pmatrix} A_{n-1} & A_n \\ B_{n-1} & B_n
\end{pmatrix} \begin{pmatrix} 0 & 1 \\ 1 & \be_{n+1} \end{pmatrix}
= \begin{pmatrix} A_n & A_{n-1} + \be_{n+1} A_n \\
B_n & B_{n-1} + \be_{n+1} B_n \end{pmatrix},
\end{equation*}
which proves \eqref{4.12} with $A_{n+1}$ and $B_{n+1}$ defined by
\eqref{4.13}.
Similarly, we prove \eqref{4.14} by induction on $n$.  The case $n=1$ is
obvious,  Assuming that \eqref{4.13} is satisfied for some $n \ge 1$,
we obtain from \eqref{4.13} that
\begin{equation*}
B_{n+1}/B_n = B_{n-1}/B_n + \be_{n+1}.
\end{equation*}
Because $\Re(w) \ge \gamma$, where $w:= B_n/B_{n-1}$, we see that
$|1/w - 1/(2\gamma)| \le 1/(2 \gamma)$ and
 $\Re(1/w) = \Re(B_{n-1}/B_n) \ge 0$, so
\begin{equation*}
\Re(B_{n+1}/B_n) \ge \Re(B_{n-1}/B_n)  + \Re(\be_{n+1}) \ge \gamma.
\end{equation*}
Hence \eqref{4.13} is satisfied for all $n \ge 1$.
Because $\det(M_j) = -1$ for all $j \ge 1$, we get that
$\det \bigl(\begin{smallmatrix} A_{n-1} & A_n \\ B_{n-1} & B_n
\end{smallmatrix}\bigr) = (-1)^n$, and \eqref{4.15} follows.
\end{proof}

Before proceeding further, it will be convenient to establish some
elementary calculus propositions.  For $(u,v) \in \R^2 \setminus\{(0,0)\}$
and $s >0$, define
\begin{equation*}
G(u,v;s) = (u^2 + v^2)^{-s}.
\end{equation*}
Define $D_1 = (\partial/\partial u)$, so $D_1^m = (\partial/\partial u)^m$
for positive integers $m$; similarly, let 
$D_2 = (\partial/\partial v)$ and $D_2^m = (\partial/\partial v)^m$.
\begin{lem}
\label{lem:4.4}
For positive integers $m$, there exist polynomials in $u$ and $v$ with
coefficients depending on $s$, $P_m(u,v;s)$ and $Q_m(u,v;s)$, such that
\begin{equation*}
D_1^m G(u,v;s) = P_m(u,v;s) G(u,v;s+m), \
D_2^m G(u,v;s) = Q_m(u,v;s) G(u,v;s+m).
\end{equation*}
Furthermore, we have $P_1(u,v;s) = -2 s u$, $Q_1(u,v;s) = -2 s v$,
and for positive integers $m$,
\begin{equation*}
P_{m+1}(u,v;s) = (u^2 + v^2) (D_1 P_m(u,v;s)) - 2(s+m) u P_m(u,v;s)
\end{equation*}
and
\begin{equation*}
Q_{m+1}(u,v;s) = (u^2 + v^2) (D_2 Q_m(u,v;s)) - 2(s+m) v Q_m(u,v;s).
\end{equation*}
\end{lem}
\begin{proof}
If $m =1$,
\begin{equation*}
D_1 G(u,v;s) = (-2 s u) \, G(u,v;s+1), \qquad 
D_2 G(u,v;s) = (-2 s v) (u^2 + v^2; s+1),
\end{equation*}
so $P_1(u,v;s) = -2 su$ and $Q_1(u,v;s) = -2 sv$. 

We now argue by induction and assume we have proved the existence
of $P_j(u,v;s)$ and $Q_j(u,v;s)$ for $1 \le j \le m$.  It follows that
\begin{multline*}
D_1^{m+1} G(u,v;s) = D_1[P_m(u,v;s) G(u,v;s+m)]
\\
 = [D_1 P_m(u,v;s)] G(u,v;s+m)] + P_m(u,v;s)[-2(s+m)u]G(u,v;s+m+1)
\\
 = [(u^2 + v^2) (D_1 P_m(u,v;s)) -2(s+m) u P_m(u,v;s)] G(u,v;s+m+1).
\end{multline*}
This proves the lemma with
\begin{equation*}
P_{m+1}(u,v;s):= (u^2 + v^2) (D_1 P_m(u,v;s)) -2(s+m) u P_m(u,v;s).
\end{equation*}
An exactly analogous argument, which we leave to the reader, shows that
\begin{equation*}
Q_{m+1}(u,v;s):= (u^2 + v^2) (D_2 Q_m(u,v;s)) -2(s+m) v Q_m(u,v;s).
\end{equation*}
\end{proof}
An advantage of working with M\"obius transformations is that one can
easily obtain tractable formulas for expressions like
$(\theta_{\be_1} \circ \theta_{\be_{2}} \cdots \circ \theta_{\be_n})(z)$. Such
formulas allow more precise estimates for the left hand side of \eqref{1.14}
than we obtained in Section 5 of \cite{hdcomp1}.

\begin{lem}
\label{lem:4.5}
In the notation of Lemma~\ref{lem:4.4}, for all $(u,v) \in \R^2 \setminus 
\{(0,0)\}$, for all $s >0$, and all positive integers $m$, 
$P_m(u,v;s) = Q_m(v,u;s)$.
\end{lem}
\begin{proof}
Fix $s >0$.  We have $P_1(u,v;s) = Q_1(v,u;s)$ for all $(u,v) \neq (0,0)$.
Arguing by mathematical induction, assume that for some positive integer $m$
we have proved that $P_m(u,v;s) = Q_m(v,u;s)$ for all $(u,v) \neq (0,0)$.
For a fixed $(u,v) \neq (0,0)$, we obtain, by virtue of the recursion
formula in Lemma~\ref{lem:4.4},
\begin{align*}
P_{m+1}(v,u;s) &= (u^2 + v^2) \lim_{\Delta v \rightarrow 0}
\frac{P_m(v+ \Delta v, u; s) - P_m(v,u;s)}{\Delta v}
\\
&\qquad -2(s+m)v P_m(v,u;s)
\\
&= (u^2 + v^2) \lim_{\Delta v \rightarrow 0}
\frac{Q_m(u,v+ \Delta v; s) - Q_m(u,v;s)}{\Delta v}
\\
&\qquad -2(s+m)v Q_m(u,v;s) 
\\
&= Q_{m+1}(u,v;s).
\end{align*}
By mathematical induction, we conclude that 
$P_n(u,v;s) = Q_n(v,u;s)$ for all positive integers $n$.
\end{proof}

\begin{remark}
\label{rem:4.6}
By using the recursion formula in Lemma~\ref{lem:4.4}, one can easily
compute $P_j(u,v;s)$ for $1\le j\le 4$.
\begin{align*}
P_1(u,v;s) &= -2s u,
\\
P_2(u,v;s) &= 2s(2s+1)u^2 - 2s v^2,
\\
P_3(u,v;s) &= -2s(2s+1)(2s+2)u^3 + (2s)(2s+2)(3) u v^2,
\\
P_4(u,v;s) &= (2s)(2s+2)[(2s+1)(2s+3)u^4 - 6(2s+3)u^2 v^2 + 3 v^4].
\end{align*}
\end{remark}
By virtue of Lemma~\ref{lem:4.5}, we also obtain formulas for 
$Q_j(v,u;s) = P_j(u,v;s)$. Also, Lemmas~\ref{lem:4.4} and \ref{lem:4.5}
imply that
\begin{equation*}
\frac{D_1^j G(u,v;s)}{G(u,v;s)} = \frac{P_j(u,v;s)}{(u^2+v^2)^j}, \qquad
\frac{D_2^j G(u,v;s)}{G(u,v;s)} = \frac{P_j(v,u;s)}{(u^2+v^2)^j}
\end{equation*}
and the latter formulas will play a useful role in this section. In particular,
for a given constant $\gamma >0$, we shall need good estimates for
\begin{equation*}
\sup \Big\{\frac{D_k^j G(u,v;s)}{G(u,v;s)}: u \ge \gamma, v \in \R\Big\}
\text{ and }
\inf \Big\{\frac{D_k^j G(u,v;s)}{G(u,v;s)}: u \ge \gamma, v \in \R\Big\}
\end{equation*}
where $k=1,2$ and $1 \le j \le 4$.  Although the arguments used to prove these
estimates are elementary, these results will play a crucial role in our
later work.
\begin{lem}
\label{lem:4.7}
Let $\gamma >0$ be a given constant and assume that $u \ge \gamma$ and $v \in
\R$.  Let $D_1 = (\partial/\partial u)$ and $G(u,v;s) = (u^2 + v^2)^{-s}$, where
$s >0$. For $j \ge 1$ we have
\begin{equation*}
\frac{D_1^j G(u,v;s)}{G(u,v;s)}= \frac{P_j(u,v;s)}{(u^2+v^2)^j},
\end{equation*}
where $P_j(u,v;s)$ is as defined in Remark~\ref{rem:4.6}; and the following
estimates are satisfied.
\begin{gather*}
-\frac{2s}{\gamma} \le \frac{D_1 G(u,v;s)}{G(u,v;s)} <0,
\\
-\frac{s}{4\gamma^2(s+1)} \le \frac{D_1^2 G(u,v;s)}{G(u,v;s)} 
\le \frac{2s(2s+1)}{\gamma^2},
\\
-\frac{2s(2s+1)(2s+2)}{\gamma^3} \le \frac{D_1^3 G(u,v;s)}{G(u,v;s)} 
\le \frac{2s(2s+2)}{\gamma^3(s+2)^2},
\\
-\frac{2s(s+1)(2s+2)(3)}{\gamma^4} \le \frac{D_1^4 G(u,v;s)}{G(u,v;s)} 
\le \frac{2s(2s+1)(2s+2)(2s+3)}{\gamma^4}.
\end{gather*}
\end{lem}
\begin{proof}
By Remark~\ref{rem:4.6}, 
\begin{equation*}
\frac{D_1^j G(u,v;s)}{G(u,v;s)}= \frac{P_j(u,v;s)}{(u^2+v^2)^j},
\end{equation*}
and Remark~\ref{rem:4.6} provides formulas for $P_j(u,v;s)$.  It follows that
\begin{equation*}
\frac{D_1^j G(u,v;s)}{G(u,v;s)} = \frac{-2su}{u^2+v^2} < 0.
\end{equation*}
Since
\begin{equation*}
\frac{2su}{u^2+v^2} \le \frac{2su}{u^2} \le \frac{2s}{\gamma},
\end{equation*}
we also see that
\begin{equation*}
\frac{D_1 G(u,v;s)}{G(u,v;s)} \ge -\frac{2s}{\gamma}.
\end{equation*}
Using Remark~\ref{rem:4.6}, we see that
\begin{equation*}
\frac{D_1^2 G(u,v;s)}{G(u,v;s)} = \frac{2s(2s+1)u^2 -2s v^2}{(u^2 + v^2)^2},
\end{equation*}
so
\begin{equation*}
\frac{D_1^2 G(u,v;s)}{G(u,v;s)} \le \frac{2s(2s+1)u^2}{(u^2 + v^2)^2}.
\end{equation*}
Since
\begin{equation*}
\frac{u^2}{(u^2+v^2)^2} \le \frac{u^2}{u^4} \le \frac{1}{\gamma^2},
\end{equation*}
we find that
\begin{equation*}
\frac{D_1^2 G(u,v;s)}{G(u,v;s)} \le \frac{2s(2s+1)}{\gamma^2},
\end{equation*}
If we write $v^2 = \rho u^2$, we see that
\begin{equation*}
\frac{D_1^2 G(u,v;s)}{G(u,v;s)} = \frac{2s(2s+1-\rho)}{u^2 (1+ \rho)^2},
\end{equation*}
and if $0 \le \rho \le 2s+1$, we obtain the upper bound given above and a
lower bound of zero. If $\rho > 2s+1$, we see that
\begin{equation*}
\frac{D_1^2 G(u,v;s)}{G(u,v;s)} \ge \frac{2s}{\gamma^2}
\inf\left\{\frac{2s+1-\rho}{(1+ \rho)^2}: \rho > 2s+1\right\}.
\end{equation*}
It is a simple calculus exercise to show that
\begin{equation*}
\inf\left\{\frac{2s+1-\rho}{(1+ \rho)^2}: \rho > 2s+1\right\}
= - \frac{1}{8(s+1)},
\end{equation*}
achieved at $\rho = 4s+3$; and this gives the lower estimate
$-s/[4\gamma^2(s+1)]$ of the lemma.

Using Remark~\ref{rem:4.6} again, we see that
\begin{equation*}
\frac{D_1^3 G(u,v;s)}{G(u,v;s)} = \frac{2s(2s+2)u[-(2s+1)u^2 + 3v^2]}
{(u^2 + v^2)^3}.
\end{equation*}
It follows that
\begin{multline*}
\frac{D_1^3 G(u,v;s)}{G(u,v;s)} \ge -2s(2s+1)(2s+2) \left[\frac{u}{(u^2 + v^2)}
\right]^3
\\
\ge -2s(2s+1)(2s+2) \left[\frac{1}{u}\right]^3
\ge -2s(2s+1)(2s+2) \frac{1}{\gamma^3}.
\end{multline*}
On the other hand, if we write $v^2 = \rho u^2$, then
\begin{multline*}
\frac{D_1^3 G(u,v;s)}{G(u,v;s)} = \frac{2s(2s+2)}{u^3}
\frac{[3 \rho -(2s+1)]}{(1+ \rho)^3}
\\
\le \frac{2s(2s+2)}{\gamma^3}
\sup \left\{\frac{3 \rho -(2s+1)}{(1+ \rho)^3}: \rho \ge 0\right\}.
\end{multline*}
Once again, a straightforward calculus argument shows that
\begin{equation*}
\sup \left\{\frac{3 \rho -(2s+1)}{(1+ \rho)^3}: \rho \ge 0\right\}
= \frac{1}{(s+2)^2}
\end{equation*}
and the supremum is achieved at $\rho = s+1$.  Using this fact, we obtain
the upper estimate of the lemma.

Finally, we obtain from Remark~\ref{rem:4.6} that
\begin{equation*}
\frac{D_1^4 G(u,v;s)}{G(u,v;s)} = \frac{2s(2s+2)[(2s+1)(2s+3)u^4
- 6(2s+3) u^2v^2 + 3 v^4]}{(u^2 + v^2)^4}.
\end{equation*}
Dropping the negative term in the numerator and observing that
$3 \le (2s+1)(2s+3)$ and $u^4 + v^4 \le (u^2 + v^2)^2$, we see that
\begin{align*}
\frac{D_1^4 G(u,v;s)}{G(u,v;s)} &\le \frac{(2s)(2s+1)(2s+2)(2s+3)(u^4+v^4)}
{(u^2 + v^2)^4}
\\
&\le \frac{(2s)(2s+1)(2s+2)(2s+3)}{(u^2 + v^2)^2}
\le \frac{(2s)(2s+1)(2s+2)(2s+3)}{\gamma^4}.
\end{align*}
On the other hand, because $-u^4 - v^4 \le -2u^2 v^2$, we obtain that
\begin{align*}
- \frac{D_1^4 G(u,v;s)}{G(u,v;s)} &\le \frac{(2s)(2s+2)
[-3 u^4 + 6(2s+3)u^2 v^2 - 3 v^4]}{(u^2 + v^2)^4}
\\
&\le \frac{3(2s)(2s+2)[-2 u^2v^2 + (4s+6)u^2 v^2]}{(u^2 + v^2)^4}
\\
&\le \frac{3(2s)(2s+2)[4(s+1)(u^2+ v^2)^2/4]}{(u^2 + v^2)^4}
\\
&\le \frac{3(2s)(2s+2)(s+1)}{(u^2 + v^2)^2} 
\le \frac{3(2s)(2s+2)(s+1)}{\gamma^4},
\end{align*}
which gives the lower estimate of Lemma~\ref{lem:4.7}.
\end{proof}

The following lemma gives analogous estimates for
\begin{equation*}
\frac{D_2^j G(u,v;s)}{G(u,v;s)} = \frac{P_j(v,u;s)}{(u^2 + v^2)^j}.
\end{equation*}
\begin{lem}
\label{lem:4.8}
Let $\gamma >0$ be a given real number, $D_2 = (\partial/\partial v)$ 
and for $s>0$ and $(u,v) \in \R^2 \setminus\{(0,0)\}$, define
$G(u,v;s) = (u^2 + v^2)^{-s}$, If $u \ge \gamma$ and $v \in \R$, we have
the following estimates.
\begin{gather*}
\frac{|D_2 G(u,v;s)|}{G(u,v;s)} \le \frac{s}{\gamma},
\\
-\frac{2s}{\gamma^2} \le \frac{D_2^2 G(u,v;s)}{G(u,v;s)} 
\le \frac{2s(2s+1)}{4\gamma^2},
\\
\frac{|D_2^3 G(u,v;s)|}{G(u,v;s)} 
\le \frac{2s(2s+2)}{\gamma^3} \max\left\{\frac{25\sqrt{5}}{72}
, \frac{2s+1}{8}\right\}
\\
-\frac{2s(s+1)(2s+2)(3)}{\gamma^4} \le \frac{D_2^4 G(u,v;s)}{G(u,v;s)} 
\le \frac{2s(2s+1)(2s+2)(2s+3)}{\gamma^4}.
\end{gather*}
\end{lem}
\begin{proof}
By Remark~\ref{rem:4.6}, $P_1(v,u;s) = -2sv$, so
\begin{equation*}
\frac{|D_2 G(u,v;s)|}{G(u,v;s)} = \frac{2s |v|}{u^2 + v^2}.
\end{equation*}
The map $w \mapsto w/(u^2+w^2)$ has its maximum on $[0, \infty)$ at $w=u$,
so $(2s|v|/(u^2+v^2) \le s/u \le s/\gamma$; and we obtain the first
inequality in Lemma~\ref{lem:4.8}.  Using Remark~\ref{rem:4.6} again, we see
that
\begin{equation*}
\frac{D_2^2 G(u,v;s)}{G(u,v;s)} = \frac{2s[(2s+1)v^2 - u^2]}{(u^2 + v^2)^2}.
\end{equation*}
It follows that
\begin{equation*}
\frac{D_2^2 G(u,v;s)}{G(u,v;s)} = 2s(2s+1)\frac{|v|^2}{(u^2 + v^2)^2}.
\end{equation*}
The map 
$v \mapsto |v|/(u^2+v^2)$ has its maximum at $|v|=u$, so $[|v|/(u^2+v^2)]^2
\le 1/(4u^2)\le 1/(4 \gamma^2)$, and
\begin{equation*}
\frac{D_2^2 G(u,v;s)}{G(u,v;s)} = \frac{2s(2s+1)}{4\gamma^2}.
\end{equation*}
Similarly, one obtains
\begin{equation*}
\frac{D_2^2 G(u,v;s)}{G(u,v;s)} \ge - \frac{2s u^2}{(u^2 + v^2)^2}
\ge - \frac{2s}{u^2} \ge - \frac{2s}{\gamma^2}.
\end{equation*}

With the  aid of Remark~\ref{rem:4.6} again, we see that
\begin{equation*}
\frac{D_2^3 G(u,v;s)}{G(u,v;s)} = 2s(2s+2)v \frac{[-(2s+1)v^2 + 3u^2]}
{(u^2+v^2)^3} := A(u,v).
\end{equation*}
For a fixed $u \ge \gamma$, $v \mapsto A(u,v)$ is an odd function of
$v$, so if
$\alpha(u) = \sup \{A(u,v): v \in \R\}$, $- \alpha(u)
= \inf\{A(u,v): v \in \R\}$. If $v \le 0$,
\begin{multline*}
A(u,v) \le (2s)(2s+1)(2s+2) \left[\frac{|v|}{u^2+v^2}\right]^3
\le (2s)(2s+1)(2s+2) \left[\frac{u}{2u^2}\right]^3
\\
\le \frac{(2s)(2s+1)(2s+2)}{8 \gamma^3}.
\end{multline*}
If $v >0$,
\begin{equation*}
A(u,v) \le (2s)(2s+2)(3u^2) \frac{v}{(u^2+v^2)^3}.
\end{equation*}
A calculation shows that $v \mapsto v/(u^2+v^2)^3$ achieves its maximum
for $v \ge 0$ at $v = u/\sqrt{5}$, so for $v >0$,
\begin{equation*}
A(u,v) \le (2s)(2s+2)(3u^{-3}) [\sqrt{5}(6/5)^3]^{-1}
\le (2s)(2s+2) \gamma^{-3} (25\sqrt{5}/72).
\end{equation*}
Note that $25\sqrt{5}/72 \approx .7764 < 1$.
Using Remark~\ref{rem:4.6} again, we see that
\begin{equation*}
\frac{D_2^4 G(u,v;s)}{G(u,v;s)} = 2s(2s+2) \frac{[(2s+1)(2s+3)v^4
-6(2s+3) u^2 v^2 + 3 u^4]}{(u^2+v^2)^4}.
\end{equation*}
Since $u^4 + v^4 \le (u^2+v^2)^2$, it follows easily that
\begin{equation*}
\frac{D_2^4 G(u,v;s)}{G(u,v;s)} \le 2s(2s+2)(2s+1)(2s+3) \frac{u^4 + v^4}
{(u^2+v^2)^4} \le 2s(2s+2)(2s+1)(2s+3) \gamma^{-4}.
\end{equation*}

Similarly, we see that
\begin{multline*}
(2s+1)(2s+3)v^4- 6(2s+3) u^2v^2+ 3 u^4 \ge
3(u^4+v^4) - 6(2s+3)[(u^2+v^2)/2]^2
\\
\ge 3 (u^2+v^2)^2 - 6 [(u^2+v^2)/2]^2 -6(2s+3) [(u^2+v^2)/2]^2.
\end{multline*}
This implies that
\begin{multline*}
\frac{D_2^4 G(u,v;s)}{G(u,v;s)} \ge 2s(2s+2) \frac{3-3/2-3/2(2s+3)}
{(u^2+v^2)^2}
\\
\ge -(2s)(2s+2)3(s+1)(u^2+v^2)^{-2} 
\ge -(2s)(2s+2)(3s+3)\gamma^{-4},
\end{multline*}
which completes the proof of Lemma~\ref{lem:4.8}. Note that 
$(2s)(2s+1)(2s+2)(2s+3) \ge 2s(2s+2)(3s+3)$.
\end{proof}
\begin{remark}
\label{rem:4.9}
Lemmas~\ref{lem:4.7} and \ref{lem:4.8} show that whenever $u \ge \gamma >0$,
$s >0$, $k=1$ or $k=2$, and $1 \le j \le 4$,
\begin{equation*}
\frac{|D_k^j G(u,v;s)|}{G(u,v;s)} \le (2s)(2s+1)\cdots (2s+j-1) \gamma^{-j}.
\end{equation*}
We have not determined whether the above inequality holds for all $j \ge 1$.
\end{remark}

Using Lemmas~\ref{lem:4.7} and \ref{lem:4.8}, we can give uniform estimates
for the quantities $(\partial/\partial x)^j v_s(x,y)/v_s(x,y)$ and
$(\partial/\partial y)^j v_s(x,y)/v_s(x,y)$, where $s >0$, $1 \le j \le 4$,
and $v_s(x,y)$ is the unique (to within normalization) strictly positive
eigenfunction of the linear operator $\Lambda_s: C^m_{\C}(\bar H) \to
C^m_{\C}(\bar H)$ in \eqref{4.2} for $m \ge 1$.

\begin{thm}
\label{thm:4.10}
Let $s$ denote a positive real and let $\B$ and $\theta_b$, $b \in \B$, be as in
(H5.1). Let $H$
be a bounded, mildly regular open subset of $\C:= \R^2$ such that $H
\supset G_{\gamma}=\{z\in \C : |z - 1/(2 \gamma)| < 1/(2
\gamma)\}$, and $\Re(z) >0$ for all $z \in H$, so $\theta_{\be}(H)
  \subset G_{\gamma}$ for all $\be \in \B$. For a positive integer
$m$, define a complex Banach space $C^m_{\C}(\bar H) = X$ and 
let $\Lambda_s: X \to X$ be defined as in \eqref{4.2}.
Then $\Lambda_s$ has a unique (to within normalization) strictly positive
eigenfunction $v_s \in X$ and $v_s \in C^{\infty}$.  Furthermore,
we have the following estimates for $(x,y) \in \bar H$.
\begin{gather}
\label{4.16}
-\frac{2s}{\gamma} \le \frac{\partial v_s(x,y)}{\partial x}
[v_s(x,y)]^{-1} \le 0,
\\
\label{4.17}
-\frac{s}{4\gamma^2(s+1)} \le \frac{\partial^2 v_s(x,y)}{\partial x^2}
[v_s(x,y)]^{-1} \le \frac{2s(2s+1)}{\gamma^2},
\\
\label{4.18}
-\frac{2s(2s+1)(2s+2)}{\gamma^3} \le \frac{\partial^3 v_s(x,y)}{\partial
x^3} [v_s(x,y)]^{-1} \le \frac{(2s)(2s+2)}{\gamma^3(s+2)^2},
\\
\label{4.19}
-\frac{2s(2s+2)(3s+3)}{\gamma^4} \le
\frac{\partial^4 v_s(x,y)}{\partial x^4} [v_s(x,y)]^{-1}
\le \frac{(2s)(2s+1)(2s+2)(2s+3)}{\gamma^4},
\\
\label{4.20}
\Big| \frac{\partial v_s(x,y)}{\partial y}\Big| [v_s(x,y)]^{-1}
\le \frac{s}{\gamma},
\\
\label{4.21}
-\frac{2s}{\gamma^2} \le \frac{\partial^2 v_s(x,y)}{\partial y^2}
[v_s(x,y)]^{-1} \le \frac{2s(2s+1)}{4\gamma^2},
\\
\label{4.22}
\Big|\frac{\partial^3 v_s(x,y)}{\partial y^3}\Big| [v_s(x,y)]^{-1} \le 
\frac{(2s)(2s+2)}{\gamma^3} \max\{25\sqrt{5}/72, (2s+1)/8\},
\\
\label{4.23}
-\frac{2s(2s+2)(3s+3)}{\gamma^4} \le
\frac{\partial^4 v_s(x,y)}{\partial y^4}[v_s(x,y)]^{-1}
\le \frac{(2s)(2s+1)(2s+2)(2s+3)}{\gamma^4}.
\end{gather}
Hence, if $D_1 = \partial/\partial x$ and $D_2 = \partial/\partial y$, we have
for $k=1,2$ and $1 \le j \le 4$ that
\begin{equation}
\label{4.24}
\frac{|D_k^j v_s(x,y)|}{v_s(x,y)}
 \le \frac{(2s)(2s+1) \cdots (2s+j-1)}{\gamma^j}.
\end{equation}
\end{thm}
\begin{proof}
For any integer $m \ge 1$, we can view $\Lambda_s$ as a bounded linear
operator of $C^m_{\C}(\bar H)$ to $C^m_{\C}(\bar H)$.  We know that $\Lambda_s$ has a
strictly positive eigenfunction $v_s(x,y) \in C_{\C}^m(\bar H)$ such that
$\sup\{v_s(x,y): (x,y) \in\bar H\} =1$.  By the uniqueness of this
eigenfunction, $v_s(x,y)$ must actually be $C^{\infty}$.

Using the notation of \eqref{4.6} and \eqref{4.7} and also using \eqref{4.15}
in Lemma~\ref{lem:4.3}, we see that
\begin{equation*}
\Big|\frac{d}{dz} \theta_{\tilde \omega}(z)\Big|^s = |B_{n-1}|^{-2s}
|z + B_n/B_{n-1}|^{-2s}.
\end{equation*}
By Lemma~\ref{lem:4.3}, $\Re(B_n/B_{n-1}) \ge \gamma_{\omega} \ge \gamma$, so
writing $\Im(B_n/B_{n-1}) = \delta_{\omega}$, we obtain that for $k=1,2$ and
$1 \le j$,
\begin{multline}
\label{4.25}
D_k^j\left(\Big|\frac{d}{dz} \theta_{\tilde \omega}(z)\Big|^s\right) 
\Big|\frac{d}{dz} \theta_{\tilde \omega}(z)\Big|^s
\\
= \Big(D_k^j\Big[(x+ \gamma_{\omega})^2 + (y+ \delta_{\omega})^2\Big]^{-s}\Big)
\Big[(x+ \gamma_{\omega})^2 + (y+ \delta_{\omega})^2\Big]^{s}.
\end{multline}
However, if we write $(x+ \gamma_{\omega}) = u \ge \gamma$ and
$(y+ \delta_{\omega}) =v$, we see that
\begin{multline}
\label{4.26}
\left(\Big(\frac{\partial}{\partial x}\Big)^j
\Big[(x+ \gamma_{\omega})^2 + (y+ \delta_{\omega})^2\Big]^{-s}\right)
\Big[(x+ \gamma_{\omega})^2 + (y+ \delta_{\omega})^2\Big]^{-s}\
\\
= \left[\Big(\frac{\partial}{\partial u}\Big)^jG(u,v;s)\right]
\left[G(u,v;s\right]^{-1},
\end{multline}
where the right hand side of the above equation is evaluated at 
$u = x+ \gamma_{\omega}$ and $v = y+ \delta_{\omega}$.  If we combine
\eqref{4.25} and \eqref{4.26} with the estimates in Lemma~\ref{lem:4.7}
and if we then use \eqref{4.11}, we obtain the estimates on
$(\partial/\partial x)^j v_s(x,y)$ given in \eqref{4.16} - \eqref{4.19}.

Similarly, we have
\begin{multline}
\label{4.27}
\left(\Big(\frac{\partial}{\partial y}\Big)^j
\Big[(x+ \gamma_{\omega})^2 + (y+ \delta_{\omega})^2\Big]^{-s}\right)
\Big[(x+ \gamma_{\omega})^2 + (y+ \delta_{\omega})^2\Big]^{-s}\
\\
= \left[\Big(\frac{\partial}{\partial v}\Big)^jG(u,v;s)\right]
\left[G(u,v;s\right]^{-1}.
\end{multline}
If we combine \eqref{4.25} and \eqref{4.27} with the estimates in
Lemma~\ref{lem:4.8} and if we then use \eqref{4.11}, we obtain the
estimates on $(\partial/\partial y)^j v_s(x,y)$ given in \eqref{4.20}
- \eqref{4.23}.
\end{proof}

\begin{remark}
\label{rem:8.3}
Let $H$, $\B$, and $\theta_b$, $b \in \B$, be as in
Theorem~\ref{thm:4.10} and let $R$ and $\alpha$ be positive reals such
that $R \ge \sup\{|b|, b \in \B\}$.  Define $\theta_0: \bar H \to \bar
H$ by $\theta_0(z) =0$ for all $z \in \bar H$ and let $L_{s,R,
  \alpha}:X:=C^m(\bar H) \to C^m(\bar H)$ be as in \eqref{Lsdefza} in
Section~\ref{sec:2dexp}.  Notice that $L_{s,R, \alpha}$ satisfies all
the hypotheses of Theorem~\ref{thm:1.1}, so all the conclusions of
Theorem~\ref{thm:1.1} hold. In particular, $L_{s,R, \alpha}$ has a
unique (to within normalization) strictly positive eigenfunction $w_s
\in C^m(\bar H)$.  Because the eigenfunction $w_s$ is unique and $m
\ge 1$ is arbitrary, $w_s \in C^m(\bar H)$ for all $m \ge 1$.

We claim that exactly the same estimates given for $v_s$ in
Theorem~\ref{thm:4.10} (i.e., \eqref{4.16} -- \eqref{4.24}) also hold
for $w_s$. To see this, define an index set $\D = \B \cup \{0\}$ and
for $z \in \bar H$, define $g_{\delta}(z) = 1/|z+ \be|^{2s}$ if
$\delta = \be \in \B$ and $g_{\delta}(z) = \alpha$ if $\delta = 0$. As
usual, if $\mu$ is a positive integer, let
\begin{equation*}
\D_{\mu} = \{ \omega = (\delta_1, \delta_2, \ldots, \delta_{\mu}) :
\delta_k \in \D \text{ for } 1 \le k \le \mu\}.  
\end{equation*}
Recall that for $\omega = (\delta_1, \delta_2, \ldots, \delta_{\mu}) \in \D_{\mu}$
and $\tilde \omega$ as in \eqref{4.6}, our convention is that
$\theta_{\tilde \omega} = \theta_{\delta_1} \circ \theta_{\delta_2}
\circ \cdots \circ \theta_{\delta_{\mu}}$ and
\begin{multline*}
g_{\tilde \omega}(z) 
\\
= g_{\delta_{\mu}}(\theta_{\delta_{\mu-1}} \circ \theta_{\delta_{\mu -2}}
\circ \cdots \circ \theta_{\delta_1}(z))
g_{\delta_{\mu-1}}(\theta_{\delta_{\mu-2}} \circ \theta_{\delta_{\mu -3}}
\circ \cdots \circ \theta_{\delta_1}(z)) \cdots
g_{\delta_2}(\theta_{\delta_1}(z)) g_{\delta_1}(z).
\end{multline*}
If $D_1 = \partial/\partial x$ and $D_2= \partial/\partial y$,
for $k \ge 1$, $p=1$ or $2$, and $z = x+ iy:=(x,y)$, we know that
\begin{equation*}
\frac{D_p^k w_s(x,y)}{w_s(x,y)} = \lim_{\mu \rightarrow \infty}
\frac{D_p^k \Big(\sum_{\w \in \D_{\mu}} g_{\tilde \w}(x,y)\Big)}{\sum_{\w \in \D_{\mu}}
  g_{\tilde \w}(x,y)}.
\end{equation*}
If $\w = (\delta_1, \delta_2, \ldots, \delta_{\mu}) \in \D_{\mu}$ and
$\delta_k \neq 0$ for $1 \le k \le \mu$, we have seen in
Lemmas~\ref{lem:4.7} and \ref{lem:4.8} that $D_p^k g_{\tilde
  \w}(x,y)/g_{\tilde \w}(x,y)$ satisfies the same estimates given for
$D_p^k v_s(x,y)/v_s(x,y)$ in equations \eqref{4.16}- \eqref{4.27}.
Thus assume that $\delta_t =0$ for some $t$, $1 \le t \le \mu$ and
$\delta_{t^{\prime}} \neq 0$ for $1 \le t^{\prime} < t$.  A little
thought shows that if $t=1$, $g_{\tilde \w}(z)$ is a positive
constant.  If $t=2$, $g_{\tilde \w}(z) = c(\w) g_{\delta_1}(z)$, where
$c(\w)$ is a positive constant.  Generally, if $2 \le t \le \mu$,
$g_{\tilde \w}(z) = c(\w)g_{\tilde \w_{t-1}}(z),$ where $c(\w)$ is a positive
constant and $\w_{t-1} = (\delta_1, \delta_2, \ldots, \delta_{t-1})
\in \D_{t-1}$ and $\delta_1, \delta_2, \ldots, \delta_{t-1} \in \B$.
It follows that $D_p^k g_{\tilde \w}(x,y)/g_{\tilde \w}(x,y) =0$ if $t=1$ and
otherwise
\begin{equation*}
D_p^k g_{\tilde \w}(x,y)/g_{\tilde \w}(x,y) 
= D_p^k g_{\tilde \w_{t-1}}(x,y)/g_{\tilde\w_{t-1}}(x,y).
\end{equation*}
By using Lemmas~\ref{lem:4.7} and \ref{lem:4.8} again, it follows
that if $\delta_t =0$ for some $t$, $1 \le t \le \mu$,
$D_p^k g_{\tilde \w}(x,y)/g_{\tilde \w}(x,y)$ is identically zero or 
satisfies the same estimates given for $v_s$ in Theorem~\ref{thm:4.10}.
Thus we see that $D_p^k w_s(x,y)/w_s(x,y)$
satisfies the same estimates given for $D_p^k v_s(x,y)/v_s(x,y)$
in Theorem~\ref{thm:4.10}.
\end{remark}

\begin{cor}
\label{cor:5.9}
Let notation and hypotheses be as in Remark~\ref{rem:8.3}.  Then $w_s$
satisfies inequalities \eqref{wrelation}--\eqref{Dyybound} in
Section~\ref{sec:2dexp}. If $\B$ and $H$ are symmetric under conjugation,
$w_s(\bar z) = w_s(z)$ for all $z \in \bar H$.
\end{cor}
\begin{proof}
  Let $H_1 \supset H$ be a convex, bounded open set such that $\Re(z)
  >0$ for all $z \in H_1$.  For $z \in \bar H_1$ and $L_{s, R,
    \alpha}$ given by \eqref{Lsdefza}, we can also view $L_{s, R,
    \alpha}$ as a bounded linear operator from $C^m_{\C}(\bar H_1) \to
  C^m_{\C}(\bar H_1)$, and this bounded linear operator has a unique strictly
  positive normalized eigenfunction $\hat w_s \in C^m_{\C}(\bar H_1)$.
  Uniqueness implies that $\hat w_s(z) = w_s(z)$ for all $z \in \bar
  H$.  Thus, after replacing $H$ by $H_1$, we can assume that $H$ is
  convex.

If $(x_1,y)$ and $(x_2,y) \in \bar H$ and $x_1 < x_2$, we obtain from
\eqref{4.16} that
\begin{equation*}
- \frac{2s}{\gamma}(x_2-x_1) \le \int_{x_1}^{x_2} \frac{\partial}{\partial x}
\log w_s(x,y) \, dx = \log \Big(\frac{w_s(x_2,y)}{w_s(x_1,y)}\Big) \le 0,
\end{equation*}
which gives \eqref{wrelation2}.  If $(x_1,y)$ and $(x_2,y) \in \bar H$ and 
$y_1 < y_2$, we obtain from \eqref{4.20} that
\begin{equation*}
- \frac{s}{\gamma}(y_2-y_1) \le \int_{y_1}^{y_2} \frac{\partial}{\partial y}
\log w_s(x,y) \, dy \le \frac{s}{\gamma}(y_2- y_1),
\end{equation*}
which gives \eqref{wrelation3}.  For $z_0$ and $z_1 \in H$,
define $z_t = (1-t)z_0 + t z_1$ and note that
\begin{multline*}
\Big|\int_0^1 \frac{d}{dt} \log(w_s(z_t)) \, dt \Big| = \Big|\log\Big(
\frac{w_s(z_1)}{w_s(z_0)}\Big)\Big|
\\
\le \int_0^1\Big|\frac{D_1 w_s(z_t)}{w_s(z_t)}(x_1-x_0)
+ \frac{D_2w_s(z_t)}{w_s(z_t)} (y_1 - y_0) \Big| \, dt,
\end{multline*}
where $z_j = (x_j,y_j)$, $j =0,1$.
Using \eqref{4.16} and \eqref{4.20}, we obtain
\begin{equation*}
\Big|\log\Big(
\frac{w_s(z_1)}{w_s(z_0)}\Big)\Big| \le 
\int_0^1 \Big| \frac{2s}{\gamma}|x_1-x_0| + \frac{s}{\gamma}|y_1 - y_0|
\Big| \, dt \le \frac{\sqrt{5} s}{\gamma} \sqrt{(x_1-x_0)^2 + (y_1 - y_0)^2},
\end{equation*}
which shows that $w_s$ satisfies \eqref{wrelation}.
Combining Remark~\ref{rem:8.3} and Corollary~\ref{cor:5.9}, we see that $w_s$
in Corollary~\ref{cor:5.9} satisfies \eqref{wrelation}--\eqref{Dyybound}.
It remains to verify the final statement in Corollary~\ref{cor:5.9}.
If $\lambda_s = r(L_{s,R,\alpha}) >0$, we know that $w_s$ is the unique normalized,
strictly positive eigenfunction of $L_{s,R,\alpha}$ with eigenvalue $\lambda_s$. Hence,
\begin{multline*}
\lambda_s w_s(\bar z) = \sum_{\be \in \B} \frac{1}{|\bar z +b|^{2s}}
w_s(1/(\bar z +b)) + \alpha w_s(0)
\\
= \sum_{\be \in \B} \frac{1}{|\bar z + \bar b|^{2s}}
w_s(1/(\bar z +\bar b)) + \alpha w_s(0).
\end{multline*}
If we define $\tilde w_s(z) = w_s(\bar z)$ for all $z \in \bar H$, the above
calculation shows
\begin{equation*}
\lambda_s \tilde w_s(z) = \sum_{\be \in \B} \frac{1}{|z +b|^{2s}}
\tilde w_s(\theta_b(z))+ \alpha \tilde w_s(0)
= \sum_{\be \in \B} \frac{1}{|z + b|^{2s}}
\tilde w_s(\theta_b(z)) + \alpha \tilde w_s(0).
\end{equation*}
By uniqueness of the strictly positive normalized eigenfunction,
this implies that $\tilde w_s = w_s$, so
$w_s(z) = w_s(\bar z)$ for all $z \in H$.
\end{proof}
It remains to consider the case that $\B$ in Theorem~\ref{thm:4.10} is
countably infinite and that $s >0$ is such that $\sum_{b \in \B} (1/|b|^{2s})
< \infty$.

\begin{thm}
\label{thm:thm5.10}
Let $\B$ be a countably infinite set such that $\B \subseteq \{z \in \C:
\Re(z) \ge \gamma \ge 1\}$.  Assume that $s >0$ is such that $\sum_{b \in \B}
(1/|b|^{2s}) < \infty$. Let $H$ and $G_{\gamma}$ be as in
Theorem~\ref{thm:4.10}. As was noted in Section~\ref{sec:2dexp} (see also
Section 5 in \cite{A} and \cite{N-P-L}), $L_s: C_{\C}(\bar H) \to
C_{\C}(\bar H)$ defines a bounded linear map, where $L_s$ is defined by
\eqref{Lsdefz}, and $L_s$ has a unique (to within scalar multiples)
strictly positive Lipschitz eigenfunction $v_s$ which satisfies inequalities
\eqref{wrelation}--\eqref{wrelation3} on $\bar H$.  If $\B$ and $H$ are
symmetric under conjugation, $v_s(\bar z) = v_s(z)$ for all $z \in \bar H$.
\end{thm}
\begin{proof}
  Select $R_0 >0$ such that $\B_{R_0}$ is nonempty,
  and for $R \ge R_0$ define $L_{s,R}$ by
\begin{equation*}
L_{s,R} = \sum_{\be \in \B_R} \frac{f(\theta_b(z))}{|z+b|^{2s}}.
\end{equation*}
By Theorem~\ref{thm:4.10}, $L_{s,R}$ has a strictly positive
$C^{\infty}$ eigenfunction $v_{s,R}$ which satisfies \eqref{wrelation}--
\eqref{Dyybound} and has sup norm one.  If $d$ denotes the diameter of $H$,
\eqref{wrelation} implies that for all $z \in H$,
\begin{equation}
\label{5.29}
v_{s,R}(z) \ge \exp[-(\sqrt{5}s/\gamma)d].
\end{equation}
Now \eqref{wrelation} implies that $z \mapsto \log(v_s(z))$ is Lipschitz with
Lipschitz constant $\sqrt{5}s/\gamma$, which is independent of $R$. Using
\eqref{5.29}, it then follows that $z \mapsto v_s(z)$ is Lipschitz on $H$ with
Lipschitz constant $C$ independent of $R \ge R_0$.  By the Ascoli-Arzela
theorem, there exists an increasing sequence of positive reals $R_j
\rightarrow \infty$ such that $v_{s,R_j}(\cdot)$ converges uniformly on $\bar
H$ to a function $v_s$.  By uniform convergence, the function $v_s$ satisfies
\eqref{5.29} on $\bar H$, is strictly positive on $\bar H$, is continuous, and
satisfies \eqref{wrelation}--\eqref{wrelation3}.  If we define $\lambda_{s,R}=
r(L_{s,R})$ for $R \ge R_0$, Lemma~\ref{cor-nonneg} implies that
$\lambda_{s,R} \le \lambda_{s,R^{\prime}}$ whenever $R \le R^{\prime}$. If we
define $M_R$ by
\begin{equation*}
M_R = \|L_{s,R}\| = \sup \Big\{\sum_{\be \in \B_R} \frac{1}{|z+b|^{2s}} : z \in
\bar H\Big\},
\end{equation*}
$r(L_{s,R}) \le M_R$ and $M_R \le M$, where
\begin{equation*}
M = \sup \Big\{\sum_{\be \in \B} \frac{1}{|z+b|^{2s}} : z \in \bar H\Big\}.
\end{equation*}
Using our assumption that $\sum_{\be \in \B} (1/|b|^{2s}) < \infty$, one
can prove that $\sum_{\be \in \B} (1/|z+b|^{2s}) < \infty$ for all $z
\in \bar H$ and that $\sum_{\be \in \B_{R_j}} (1/|z+b|^{2s})$ converges
uniformly on $\bar H$ to $\sum_{\be \in \B} (1/|z+b|^{2s})$ as $j \rightarrow
\infty$, so $z \mapsto \sum_{\be \in \B} (1/|z+b|^{2s})$ is continuous
and bounded on $\bar H$ and $M < \infty$. Since $\lambda_{s,R_j}$ is an
increasing sequence which is bounded by $M$, $\lambda_{s,R_j} \rightarrow
\lambda_s >0$. Using this information one can see that
$\sum_{\be \in \B_{R_j}} \big[v_{s,R_j}(\theta_b(z))/|z+b|^{2s}\big]$ converges
uniformly on $\bar H$ to $\sum_{\be \in \B} \big[v_{s}(\theta_b(z))/|z+b|^{2s}
\big] = \lambda_s v_s(z)$.  Details are left to the reader.

Because $v_s$ is a strictly positive eigenfunction on $\bar H$ for
$L_s$ with eigenvalue $\lambda_s$, Lemma~\ref{lem:nonneg} implies that
$\lambda_s = r(L_s)$. Theorem 5.3 in \cite{A} implies that $L_s$ has
no complex eigenvalues $\lambda \neq r(L_s)$ with $|\lambda| = r(L_s)$.
If $\B$ and $H$ are symmetric under conjugation, it was proved in
Corollary~\ref{cor:5.9} that $v_{s,R_j}(\bar z) = v_{s,R_j}(z)$ for
all $z \in H$.  The corresponding result for $v_s$ follows by letting $R_j
\rightarrow \infty$.
\end{proof}

The operator $L_s$ induces a corresponding operator $\Lambda_s: C^{0,1}(\bar
H) \to C^{0,1}(\bar H)$, where $C^{0,1}(\bar H)$ denotes the Banach space of
Lipschitz continuous maps $f:\bar H \to \C$.  One finds (see \cite{A}) that
$r(\Lambda_s) = r(L_s):= r >0$ and there exists $r^{\prime} < r$ such that
$|\zeta| \le r^{\prime}$ for all $\zeta \in \sigma(\Lambda_s)$, $\zeta \neq
r(\Lambda_s)$. However, $r(L_s)$ may fail to be an isolated point in the
spectrum of $L_s: C(\bar H) \to C(\bar H)$, even for simple examples.

\begin{thm}
\label{thm:thm5.11}
Let hypotheses and notation be as in Theorem~\ref{thm:thm5.10}. For a
given number $R >2$ and for $\B_R^{\prime}:= \{b \in \B: |b| > R\}$,
assume that there exist $\delta_{s,R} >0$ and $\eta_{s,R} \ge 0$ such
that
\begin{equation*}
\eta_{s,R} v_s(0) \le \sum_{\be \in \B_R^{\prime}} \frac{1}{|z+b|^{2s}} 
v_s(\theta_b(z)) \le \delta_{s,R} v_s(0).
\end{equation*}
Let $L_{s,R, \alpha}$ be defined by \eqref{Lsdefza} and define
$L_{s,R+} = L_{s,R, \alpha}$ for $\alpha = \delta_{s,R}$ and
$L_{s,R-} = L_{s,R, \alpha}$ for $\alpha = \eta_{s,R}$. Then we have
\begin{equation}
\label{5.31}
r(L_{s,R-}) \le r(L_s) \le r(L_{s,R+}).
\end{equation}
\end{thm}
\begin{proof}
By our assumptions, if $\lambda_s:= r(L_s)$,
\begin{equation*}
L_s v_s = \lambda_s v_s \le L_{s,R+} v_s \qquad
\text{and} \qquad
L_{s,R-} v_s \le  \lambda_s v_s.
\end{equation*}
Since $v_s$ is strictly positive on $\bar H$, Lemma~\ref{lem:nonneg}
implies \eqref{5.31}.
\end{proof}

Now that we know the strictly positive eigenfunction $v_s$ satisfies
\eqref{wrelation}--\eqref{wrelation3}, when $\B$ is countably infinite,
we can give estimates for the quantities $\delta_{s,R}$ and
$\eta_{s,R}$ in Section~\ref{sec:2dexp}.

\begin{thm}
\label{thm:thm5.12}
Assume that $\B = I_1$ or $\B=I_2$ and let $v_s$ be the unique strictly
positive eigenfunction of $L_s$ in \eqref{Lsdefz}, where we take $\bar U
\supset D$ such that $0 \le x \le 1$ and $|y| \le 1/2$ for all $(x,y) \in \bar
U$.  Assume that $s >1$ and $R >2$.  Then we have the following estimates:
\begin{multline*}
\sum_{\be \in I_1, |\be| > R} \frac{1}{|z+\be|^{2s}} v_s(\theta_{\be}(z))
\le \exp\Big(\frac{s}{\sqrt{R^2-R}}\Big)\Big(\frac{R}{R-1}\Big)^s
\\
\cdot \left[\Big(\frac{1}{2s-1}\Big)\Big(\frac{1}{R-1}\Big)^{2s-1}
+ \Big(\frac{\pi}{2}\Big)\Big(\frac{1}{s-1}\Big)
\Big(\frac{1}{R-\sqrt{2}}\Big)^{2s-2}
\right]v_s(0).
\end{multline*}
\begin{multline*}
\sum_{\be \in I_2, |\be| > R} \frac{1}{|z+\be|^{2s}} v_s(\theta_{\be}(z))
\le \exp\Big(\frac{s}{\sqrt{R^2-R}}\Big)\Big(\frac{R}{R-1}\Big)^s
\\
\cdot \left[\Big(\frac{\pi}{4}\Big)\Big(\frac{1}{s-1}\Big)
\Big(\frac{1}{R-\sqrt{2}}\Big)^{2s-2}
\right]v_s(0).
\end{multline*}
\end{thm}
\begin{proof}
First assume $\B = I_1$ in \eqref{Lsdefz}. 
Using \eqref{wrelation2} and \eqref{wrelation3}, we have
\begin{equation*}
v_s(\theta_{\be}(z)) \le \exp(s |\theta_{\be}(z)|) v_s(0).
\end{equation*}
Now for $ z = x + \irm y \in D_h$ and $\be = m+ \irm n \in I_1$, we have
\begin{multline*}
\min_{(x,y) \in D_h} (x+m)^2 + (y+n)^2
\ge \min_{0 \le x \le 1}(x+m)^2 + \min_{|y| \le 1/2} (y+n)^2
\\
\ge m^2 + (|n|-1/2)^2 \ge m^2 + n^2 - |n|.
\end{multline*}
Hence, for $z \in D_h$,
\begin {equation*}
\frac{1}{|z+\be|^2} = \frac{1}{(x+m)^2 + (y+n)^2}
\le \frac{1}{m^2 + n^2 - |n|}.
\end{equation*}
Also, it is easy to check that if $m^2 + n^2 \ge R^2 >1$,
\begin{equation*}
\frac{1}{m^2 + n^2 - |n|} \le \frac{R}{R-1} \frac{1}{m^2 + n^2}
\le \frac{1}{R^2-R}.
\end{equation*}
Hence, for $m^2 + n^2 \ge R^2 >1$ and $z \in D_h$,
\begin{equation*}
\exp(s |\theta_{\be}(z)|) \le \exp\Big(\frac{s}{\sqrt{m^2 + n^2 - |n|}}\Big)
\le \exp\Big(\frac{s}{\sqrt{R^2-R}}\Big).
\end{equation*}
It follows that
\begin{multline*}
\sum_{\be \in I_1, |\be| > R} \frac{1}{|z+\be|^{2s}} v_s(\theta_{\be}(z))
\\
\le \exp\Big(\frac{s}{\sqrt{R^2-R}}\Big) \Big(\frac{R}{R-1}\Big)^s
\Big(\sum_{\be \in I_1, |\be| > R} \Big(\frac{1}{m^2 + n^2}\Big)^s\Big) v_s(0).
\end{multline*}
Now for $n=0$ and $m \ge R$,
\begin{equation*}
\sum_{m \ge R} \frac{1}{m^{2s}} \le \int_{R-1}^{\infty} \frac{1}{r^{2s}} \, dr
= \frac{1}{2s-1} \Big(\frac{1}{R-1}\Big)^{2s-1}.
\end{equation*}
For $\be = m+ \irm n \in I_1$ with $m \ge 1$, $n\ge 1$, and $|\be| \ge R$, let
\begin{equation*}
B(m,n) = \{(\xi,\eta): m \le \xi \le m +1, n \le \eta \le n+1\}.
\end{equation*}
Then for $(u,v) \in B(m,n)$,
\begin{equation*}
\frac{1}{(u-1)^2 + (v-1)^2} \ge \frac{1}{m^2 + n^2}.
\end{equation*}
Also,
\begin{multline*}
(u-1)^2 + (v-1)^2 \ge (m-1)^2 + (n-1)^2 = m^2 + n^2 -2 (m+n) + 2
\\
\ge m^2 + n^2 -2 \sqrt{2} \sqrt{m^2 + n^2} + 2 = (\sqrt{m^2 + n^2} -
\sqrt{2})^2 \ge (R- \sqrt{2})^2 \equiv R_1^2.
\end{multline*}
Hence,
\begin{multline*}
\sum_{\substack{m \ge 1, n \ge 1 \\ m^2 + n^2 > R^2}}\ 
\Big(\frac{1}{m^2 + n^2}\Big)^s 
\le \sum_{\substack{m \ge 1, n \ge 1 \\ m^2 + n^2 > R^2}}\ 
\iint\limits_{B(m,n)} \Big(\frac{1}{(u-1)^2 + (v-1)^2}\Big)^s \, du \, dv
\\
\le \iint\limits_{\substack{u \ge 0, v \ge0 \\ u^2 + v^2 \ge R_1^2} }
\Big(\frac{1}{u^2 + v^2}\Big)^s \, du \, dv
= \frac{\pi}{2} \int_{R_1}^{\infty} \frac{1}{r^{2s}} r \, dr
= \frac{\pi}{2} \frac{r^{2-2s}}{2-2s}\Big|_{R_1}^{\infty}
\\
=\frac{\pi}{2} \frac{1}{2s-2}\frac{1}{R_1^{2s-2}}
= \frac{\pi}{4}\frac{1}{s-1}\left(\frac{1}{R- \sqrt{2}}\right)^{2s-2}.
\end{multline*}
A similar argument shows that
\begin{equation}
\label{negn}
\sum_{\substack{m \ge 1, n \le -1 \\ m^2 + n^2 > R^2}}\ 
\Big(\frac{1}{m^2 + n^2}\Big)^s 
\le \frac{\pi}{4}\frac{1}{s-1}\Big(\frac{1}{R- \sqrt{2}}\Big)^{2s-2}.
\end{equation}
Combining these estimates, we obtain
\begin{multline*}
\sum_{\be \in I_1, |\be| > R} \frac{1}{|z+\be|^{2s}} v_s(\theta_{\be}(z))
\le \exp\Big(\frac{s}{\sqrt{R^2-R}}\Big) \Big(\frac{R}{R-1}\Big)^s
\\
\cdot \left[ \frac{1}{2s-1} \Big(\frac{1}{R-1}\Big)^{2s-1}
\hskip-5pt
+ \frac{\pi}{2}\frac{1}{s-1}\Big(\frac{1}{R- \sqrt{2}}\Big)^{2s-2}
\right] v_s(0) := \delta_{s,R} v_s(0) .
\end{multline*}
The estimate for the sum over $I_2$ follows by a similar but simpler argument,
since only the inequality in \eqref{negn} is needed.
\end{proof}



\begin{remark}
\label{rem:5.4}
If $\B \subset I_1$ is an infinite set, $s > \tau(\B)$ and $v_s$ is the
corresponding strictly positive eigenfunction of $L_s$ in \eqref{Lsdefz}, an
examination of the proof of Theorem~\ref{thm:thm5.12} shows that
\begin{equation*}
\sum_{\be \in \B, |\be| > R} \frac{1}{|z+\be|^{2s}} v_s(\theta_{\be}(z))
\le \exp\Big(\frac{s}{\sqrt{R^2-R}}\Big) \Big(\frac{R}{R-1}\Big)^s
\Big(\sum_{\be \in \B, |\be| > R} \frac{1}{|b|^{2s}}\Big) v_s(0),
\end{equation*}
so an estimate for $\delta_{s,R}$ in this case will follow from an upper bound
on $\sum_{\substack{\be \in \B \\ |\be| > R}} \frac{1}{|b|^{2s}}$.
\end{remark}

It remains to estimate $\eta_{s,R}$ in Theorem~\ref{thm:thm3.3}.  We could,
of course, take $\eta_{s,R}=0$, but we can do slightly better.  Since the
argument is similar to that in Theorem~\ref{thm:thm5.12}, we just sketch
the proof.
\begin{thm}
\label{thm:thm5.13}
Assume that $\B$ is an infinite subset of $I_1$, that $s > \tau(\B)$, and
that $v_s$ is the strictly positive eigenfunction of $L_s$ in \eqref{Lsdefz},
where we take $U \supset D$ such that $0 \le x \le 1$ and $|y| \le 1/2$
for all $(x,y) \in \bar U$.  Then we have that
\begin{multline*}
\sum_{\substack{\be \in \B \\ |\be| > R}} \frac{1}{|z+\be|^{2s}} v_s(\theta_{\be}(z))
\\
\ge \exp\Big(\frac{-\sqrt{5}s}{\sqrt{R^2-R}}\Big)\Big(\frac{R}{R + \sqrt{5} +
  [5/(4R)]}\Big)^s v_s(0) \sum_{\substack{\be \in \B \\ |\be| > R}} \frac{1}{|\be|^{2s}}
\\
:= C(R,s) v_s(0) \sum_{\be \in \B, |\be| > R}  \frac{1}{|\be|^{2s}}.
\end{multline*}
If $\B = I_1$, $s >1$ and $\theta_R = \arcsin(1/(R+ \sqrt{2}))$,
\begin{multline*}
\sum_{\substack{\be \in I_1 \\ |\be| > R}} 
\frac{1}{|z+\be|^{2s}} v_s(\theta_{\be}(z))
\\
\ge C(R,s) v_s(0)(\pi- 2 \theta_R) \Big(\frac{1}{2s-2}\Big)
\Big(\frac{1}{R+ \sqrt{2}}\Big)^{2s-2} := \eta_{s,R} v_s(0).
\end{multline*}
If $\B = I_2$ and $s >1$,
\begin{multline*}
\sum_{\substack{\be \in I_2 \\ |\be| > R}} 
\frac{1}{|z+\be|^{2s}} v_s(\theta_{\be}(z))
\\
\ge C(R,s) v_s(0)(\pi/2 - 2 \theta_R) \Big(\frac{1}{2s-2}\Big)
\Big(\frac{1}{R+ \sqrt{2}}\Big)^{2s-2} := \eta_{s,R} v_s(0).
\end{multline*}
\end{thm}
\begin{proof}
By using \eqref{wrelation} and the estimate in the proof of
Theorem~\ref{thm:thm5.12} that $1/|z+b|^2 \le 1/(R^2-R)$ for $|b| \ge R$ and
$z \in \bar U$, we get
\begin{equation*}
\sum_{\substack{\be \in \B \\ |\be| > R}} \frac{1}{|z+\be|^{2s}} 
v_s(\theta_{\be}(z))
\ge \exp\Big(\frac{-\sqrt{5}s}{\sqrt{R^2-R}}\Big) v_s(0) 
\sum_{\be \in \B} \frac{1}{|z + \be|^{2s}}.
\end{equation*}
If $b \in \B$, $|b| >R$, and $z \in \bar U$, one can check that
\begin{equation*}
|z+b|^2 \le \big[|b|^2(4R^2 + 4\sqrt{5}R + 5)\big]/[4R^2],
\end{equation*}
and this gives the first inequality in Theorem~\ref{thm:thm5.13}.  If
$b = m + n \irm \in I_1$, let $\hat b = (m +1) + (n +1) \irm$ if $n
\ge 0$ and $\hat b = (m +1) + (n -1) \irm$ if $n < 0$. Let $G_R =
\{(x,y) \in \R^2: x >1 \text{ and } \sqrt{x^2 + y^2} \ge R +
\sqrt{2}\}$.  One can check that
\begin{equation*}
\sum_{\substack{\be \in I_1 \\ |\be| > R}} \frac{1}{|\be|^{2s}}
\ge \sum_{\substack{\be \in I_1 \\ |\hat \be| > R + \sqrt{2}}}
\frac{1}{|\be|^{2s}}
\ge \int_{G_R} \Big(\frac{1}{x^2 +y^2}\Big)^s \, dx \, dy,
\end{equation*}
and using polar coordinates gives the second inequality in
Theorem~\ref{thm:thm5.13}.
For $I_2$, let $H_R = \{(x,y) \in \R^2: x >1, y<-1, \text{ and } 
\sqrt{x^2 + y^2} > R + \sqrt{2}\}$.
One can check that
\begin{equation*}
\sum_{\substack{\be \in I_2 \\ |\be| > R}} \frac{1}{|\be|^{2s}}
\ge \sum_{\substack{\be \in I_2 \\ |\hat \be| > R + \sqrt{2}}}
\frac{1}{|\be|^{2s}}
\ge \int_{H_R} \Big(\frac{1}{x^2 +y^2}\Big)^s \, dx \, dy,
\end{equation*}
and one obtains the final inequality in
Theorem~\ref{thm:thm5.13} with the aid of polar coordinates.
\end{proof}
Once the mesh size $h$ has been chosen and $R>2$ has been chosen (if $\B
\subset I_1$ is infinite), the above results give formulas for nonnegative
square matrices $A_s$ and $B_s$ such that $r(A_s) \le r(L_s) \le r(B_s)$,
where $L_s$ is as in \eqref{Lsdefz}.  In particular, for $\B = I_1$, $I_2$, or
$I_3$, if $r(A_{s_2}) >1$ and $r(A_{s_2})$ is very close to one and
$r(B_{s_1}) <1$ and $r(B_{s_1})$ is very close to one, then the Hausdorff
dimension $s_*$ of the invariant set corresponding to $\B$ satisfies $s_2 <
s_* < s_1$.  Here $s_2$ and $s_1$ are obtained as described earlier.

\begin{remark}
\label{rem:5.5}
For the set $I_1$ and $s = 1.86$, evaluating the above expressions gives for
$\delta_{s,R}$ and $\eta_{s,R}$ the values
\begin{align*}
&R = 100: \delta_{s,R} = .00071, \quad 
R = 200:  \delta_{s,R} = .00021, \quad
R = 300: \delta_{s,R} =  .00010,
\\
&R = 100: \eta_{s,R} = .00059, \quad 
R = 200: \eta_{s,R} =  .00019, \quad
R = 300: \eta_{s,R} =  .000096.
\end{align*}
For the set $I_2$ and $s = 1.49$, evaluating the above expressions gives for
$\delta_{s,R}$ and $\eta_{s,R}$ the values
\begin{align*}
&R = 100: \delta_{s,R} = .0184, \quad 
R = 200:  \delta_{s,R} = .0091, \quad
R = 300: \delta_{s,R} =  .0061,
\\
&R = 100: \eta_{s,R} = .0160, \quad 
R = 200: \eta_{s,R} =  .0085, \quad
R = 300: \eta_{s,R} =  .0058.
\end{align*}
\end{remark}

\section{Computing the Spectral Radius of $A_s$ and $B_s$}
\label{sec:compute-sr}
In previous sections, we have constructed matrices $A_s$ and $B_s$ such that
$r(A_s) \le r(L_s) \le r(B_s)$.  The $m \times m$ matrices $A_s$ and $B_s$
have nonnegative entries, so the Perron-Frobenius theory for such matrices
implies that $r(B_s)$ is an eigenvalue of $B_s$ with corresponding
nonnegative eigenvector, with a similar statement for  $A_s$. One might
also hope that standard theory (see \cite{D}) would imply that $r(B_s)$,
respectively $r(A_s)$, is an eigenvalue of $B_s$ with algebraic multiplicity
one and that all other eigenvalues $z$ of $B_s$ (respectively, of $A_s$)
satisfy $|z| < r(B_s)$ (respectively, $|z| < r(A_s)$). Indeed, this would
be true if $B_s$ were {\it primitive}, i.e., if $B_s^k$ had all positive
entries for some integer $k$.  However, typically $B_s$ has many zero
columns and $B_s$ is neither primitive nor {\it irreducible} (see \cite{D});
and the same problem occurs for $A_s$.  Nevertheless, the desirable spectral
properties mentioned above are satisfied for both $A_s$ and $B_s$. Furthermore
$B_s$ has an eigenvector $w_s$ with all positive entries and with eigenvalue
$r(B_s)$; and if $x$ is any $m \times 1$ vector with all positive entries,
\begin{equation*}
\lim_{k \rightarrow \infty} \frac{B_s^k(x)}{\|B_s^k(x)\|} =
  \frac{w_s}{\|w_s\|},
\end{equation*}
where the convergence rate is geometric.  Of course, corresponding results
hold for $A_s$.  Such results justify standard numerical algorithms for
approximating $r(B_s)$ and $r(A_s)$.

These results were proved in the one dimensional case in \cite{hdcomp1}.
Similar theorems can be proved in the two dimensional case, but because the
proofs are similar, we omit the argument in the two dimensional case.  The
basic point, however, is simple: Although $A_s$ and $B_s$ both map the cone
$K$ of nonnegative vectors in $\R^m$ into itself, $K$ is {\it not} the natural
cone in which such matrices should be studied.  Instead, one proceeds by
defining, for large positive real $M$, a cone $K_M \subset K$ such that
$A_s(K_M) \subset K_M$ and $B_s(K_M) \subset K_M$.  The cone $K_M$ is the
discrete analogue of a cone which has been used before in the infinite
dimensional case (see \cite{N-P-L}, Section 5 of \cite{A}, Section 2 of \cite
{F} and \cite{Bumby1}).  Once one shows that $A_s(K_M) \subset K_M$ and
$B_s(K_M) \subset K_M$, the desired spectral properties of $A_s$ and $B_s$
follow easily.  In a later paper, we shall consider higher order piecewise
polynomial approximations to the positive eigenfunction $v_s$ of $L_s$.  We
hope to show that although the corresponding matrices $A_s$ and $B_s$ no
longer have all nonnegative entries, it is still possible to obtain rigorous
upper and lower bounds on the Hausdorff dimension.

\section{Log convexity of the spectral radius of 
$\Lambda_s$}
\label{sec:logconvex}
For $s \in \R$, we define $\Lambda_s: X \to X :=C^m(\bar H)$ 
and $L_s:Y \to Y := C(\bar H)$ by
\begin{equation}
\label{2.1}
(\Lambda_s(f))(x) = \sum_{\beta \in \B} (g_{\beta}(x))^s f (\theta
_{\beta}(x))
\end{equation}
and
\begin{equation}
\label{2.2}
(L_s(f))(x) = \sum_{\beta \in \B} (g_{\beta}(x))^s f (\theta
_{\beta}(x)).
\end{equation}

In general, if $V$ is a convex subset of a vector space $X$, we shall call
a map $f:V \to [0, \infty)$ log convex if (i) $f(x) = 0$ for all $x \in V$
or (ii) $f(x) >0$ for all $x \in V$ and $x \mapsto \log(f(x))$ is convex.
Products of log convex functions are log convex, and H\"olders inequality
implies that sums of log convex functions are log convex.

The main result of this section is the following theorem.
\begin{thm}
\label{thm:2.1}
Assume that hypotheses (H4.1), (H4.2), and (H4.3) are satisfied with $m \ge 1$
and that $H \subset \R^n$ is a bounded, open mildly regular set. For $s \in
\R$, let $\Lambda_s$ and $L_s$ be defined by \eqref{2.1} and \eqref{2.2}.
Then we have that $s \mapsto r(\Lambda_s)$ is log convex, i.e., $s
\mapsto log(r(\Lambda_s))$ is convex on $[0,\infty)$.
\end{thm}
The proof is essentially the same as the proof of 
Theorem 8.1 in \cite{hdcomp1}, so we do not repeat it here.

Results related to Theorem~\ref{thm:2.1} can be found in \cite{CC},
\cite{DD}, \cite{EE}, \cite{FF}, \cite{GG}, and \cite{HH}. Note that
the terminology {\it super convexity} is used to denote log convexity
in \cite{DD} and \cite{EE}, presumably because any log convex function
is convex, but not conversely.  Theorem~\ref{thm:2.1}, while adequate for
our immediate purposes, can be greatly generalized by a different argument
that does not require existence of strictly positive eigenvectors.  This
generalization (which we omit) contains Kingman's matrix log convexity
result in \cite{EE} as a special case.

In our applications, the map $s \mapsto r(L_s)$ will usually be strictly
decreasing on an interval $[s_1,s_2]$ with $r(L_{s_1}) >1$ and
$r(L_{s_2}) <1$, and we wish to find the unique $s_* \in (s_1,s_2)$ such that
$r(L_{s_*}) =1$.  The following hypothesis insures that $s \mapsto r(L_s)$
is strictly decreasing for all $S$.

\noindent (H7.1): Assume that $g_{\beta}(\cdot)$, $\beta \in \B$ satisfy the
conditions of (H4.1).  Assume also that there exists an integer $\mu \ge 1$
such that $g_\w(x) <1$ for all $\w \in \B_{\mu}$ and all $x \in \bar H$.

\begin{thm}
\label{thm:2.2}
Assume hypotheses (H4.1), (H4.2), (H4.3), and (H7.1) and let $H$ be
mildly regular. Then the map $s \mapsto r(\Lambda_s)$, $s \in \R$, is
strictly decreasing and real analytic and $\lim_{s \rightarrow \infty}
r(\Lambda_s) =0$.
\end{thm}
This result is also proved in \cite{hdcomp1}, so we do not repeat the
proof here.

\begin{remark}
\label{rem:2.4}
Assume that the assumptions of Theorem~\ref{thm:2.2} are satisfied and define
$\psi(x) = \log(r(L_s)) = \log(r(\Lambda_s))$ (where $\log$ denotes the
natural logarithm), so $s \mapsto \psi(s)$ is a convex, strictly
decreasing function with $\psi(0) >1$ (unless $|\B| = p =1$) and
$\lim_{s \rightarrow \infty} \psi(s) = - \infty$. We are interested in finding
the unique value of $s$ such that $\psi(s) =0$. In general suppose that
$\psi:[s_1,s_2] \to \R$ is a continuous, strictly decreasing, convex
function such that $\psi(s_1) >0$ and $\psi(s_2) <0$, so there exists
a unique $s =s_* \in (s_1,s_2)$ with  $\psi(s_*) =0$.  If $t_1$ and $t_2$ are
chosen so that $s_1 \le t_1 < t_2 \le s_*$ and $t_{k+1}$ is obtained from
$t_{k-1}$ and $t_k$ by the secant method, an elementary argument show that
$\lim_{k \rightarrow \infty} t_k = s_*$.  If $s_* \le t_2 < t_1 < s_2$ and
$s_1 \le t_3$, a similar argument shows that $\lim_{k \rightarrow \infty} t_k
= s_*$. If $\psi \in C^3$, elementary numerical analysis implies
that the rate of convergence is faster than linear ($= (1 + \sqrt{5})/2)$.
In our numerical work, we apply these observations, not directly to $\psi(s) =
\log(r(\Lambda_s))$, but to decreasing functions which closely
approximate $\log(r(\Lambda_s))$.
\end{remark}

One can also ask whether the maps $s \mapsto r(B_s)$ and $s \mapsto r(A_s)$
are log convex, where $A_s$ and $B_s$ are the previously described
approximating matrices for $L_s$.  An easier question is whether the map $s
\mapsto r(M_s)$ is log convex, where $A_s$ and $B_s$ are obtained from $M_s$
by adding error correction terms.  In \cite{hdcomp1}, it was proved that in
the one dimensional case, $s \mapsto r(M_s)$ is log convex. The proof in the
two dimensional case is similar, and we do not repeat it here.

\frenchspacing

\def\cprime{$'$} \def\cprime{$'$}
\providecommand{\bysame}{\leavevmode\hbox to3em{\hrulefill}\thinspace}
\providecommand{\MR}{\relax\ifhmode\unskip\space\fi MR }
\providecommand{\MRhref}[2]{%
  \href{http://www.ams.org/mathscinet-getitem?mr=#1}{#2}
}
\providecommand{\href}[2]{#2}

\end{document}